\documentclass{cmslatex}
\usepackage{latexsym, amssymb, enumerate, amsmath}

\usepackage{color}
\def\change#1{#1}

\begin{document}

\newcommand{\cG}{\cal G}

\newcommand{\cC}{\cal C}

\newcommand{\cB}{\cal B}

\newcommand{\cU}{\cal U}

\newcommand{\cS}{\cal S}

\newcommand{\cP}{\cal P}

\newcommand{\cE}{\cal E}

\newcommand{\cF}{\cal F}

\newcommand{\cD}{\cal D}

\newcommand{\cM}{\cal M}

\newcommand{\ra}{\rightarrow}

\newcommand{\la}{\leftarrow}

\newcommand{\Lra}{\Leftrightarrow}

\newcommand{\Ra}{\Rightarrow}

\newcommand{\BQ}{\Bbb Q}

\newcommand{\BZ}{\Bbb Z}

\newcommand{\BC}{\Bbb C}

\newcommand{\BR}{\Bbb R}

\newcommand{\BN}{\Bbb N}

\newcommand{\df}{\displaystyle \frac}

\newcommand{\lang}{\langle}

\newcommand{\rang}{\rangle}

\newcommand{\ions}{\rm{\scriptstyle ions}}

\newcommand{\ion}{\rm{\scriptstyle ion}}

\newcommand{\dip}{\rm{\scriptstyle dip}}

\newcommand{\pol}{\rm{\scriptstyle pol}}

\newcommand{\mg}{\rm{\scriptstyle mag}}

\newcommand{\free}{\rm{\scriptstyle free}}

\newcommand{\be}{\begin{equation}}

\newcommand{\ee}{\end{equation}}

\newcommand{\bma}{\mbox{\boldmath{$a$}}}

\newcommand{\bmb}{\mbox{\boldmath{$b$}}}

\newcommand{\bmc}{\mbox{\boldmath{$c$}}}

\newcommand{\bmd}{\mbox{\boldmath{$d$}}}

\newcommand{\bme}{\mbox{\boldmath{$e$}}}

\newcommand{\bmf}{\mbox{\boldmath{$f$}}}

\newcommand{\bmg}{\mbox{\boldmath{$g$}}}

\newcommand{\bmh}{\mbox{\boldmath{$h$}}}

\newcommand{\bmi}{\mbox{\boldmath{$i$}}}

\newcommand{\bmj}{\mbox{\boldmath{$j$}}}

\newcommand{\bmk}{\mbox{\boldmath{$k$}}}

\newcommand{\bml}{\mbox{\boldmath{$l$}}}

\newcommand{\bmm}{\mbox{\boldmath{$m$}}}

\newcommand{\bmn}{\mbox{\boldmath{$n$}}}

\newcommand{\bmo}{\mbox{\boldmath{$o$}}}

\newcommand{\bmp}{\mbox{\boldmath{$p$}}}

\newcommand{\bmq}{\mbox{\boldmath{$q$}}}

\newcommand{\bmr}{\mbox{\boldmath{$r$}}}

\newcommand{\bms}{\mbox{\boldmath{$s$}}}

\newcommand{\bmt}{\mbox{\boldmath{$t$}}}

\newcommand{\bmu}{\mbox{\boldmath{$u$}}}

\newcommand{\bmv}{\mbox{\boldmath{$v$}}}

\newcommand{\bmw}{\mbox{\boldmath{$w$}}}

\newcommand{\bmx}{\mbox{\boldmath{$x$}}}

\newcommand{\bmy}{\mbox{\boldmath{$y$}}}

\newcommand{\bmz}{\mbox{\boldmath{$z$}}}

\newcommand{\bmA}{\mbox{\boldmath{$A$}}}

\newcommand{\bmB}{\mbox{\boldmath{$B$}}}

\newcommand{\bmC}{\mbox{\boldmath{$C$}}}

\newcommand{\bmD}{\mbox{\boldmath{$D$}}}

\newcommand{\bmE}{\mbox{\boldmath{$E$}}}

\newcommand{\bmF}{\mbox{\boldmath{$F$}}}

\newcommand{\bmG}{\mbox{\boldmath{$G$}}}

\newcommand{\bmH}{\mbox{\boldmath{$H$}}}

\newcommand{\bmI}{\mbox{\boldmath{$I$}}}

\newcommand{\bmJ}{\mbox{\boldmath{$J$}}}

\newcommand{\bmK}{\mbox{\boldmath{$K$}}}

\newcommand{\bmL}{\mbox{\boldmath{$L$}}}

\newcommand{\bmM}{\mbox{\boldmath{$M$}}}

\newcommand{\bmN}{\mbox{\boldmath{$N$}}}

\newcommand{\bmO}{\mbox{\boldmath{$O$}}}

\newcommand{\bmP}{\mbox{\boldmath{$P$}}}

\newcommand{\bmQ}{\mbox{\boldmath{$Q$}}}

\newcommand{\bmR}{\mbox{\boldmath{$R$}}}

\newcommand{\bmS}{\mbox{\boldmath{$S$}}}

\newcommand{\bmT}{\mbox{\boldmath{$T$}}}

\newcommand{\bmU}{\mbox{\boldmath{$U$}}}

\newcommand{\bmV}{\mbox{\boldmath{$V$}}}

\newcommand{\bmW}{\mbox{\boldmath{$W$}}}

\newcommand{\bmX}{\mbox{\boldmath{$X$}}}

\newcommand{\bmY}{\mbox{\boldmath{$Y$}}}

\newcommand{\bmZ}{\mbox{\boldmath{$Z$}}}

\newcommand{\bmchi}{\mbox{\boldmath{$\chi$}}}

\newcommand{\pp}{\prime\prime}

\newcommand{\bc}{\begin{center}}

\newcommand{\ec}{\end{center}}

\newcommand{\bmcalM}{\mbox{\boldmath{${\cal M}$}}}

\newcommand{\ppp}{\prime\prime\prime}

\newcommand{\aup}{\uparrow}

\newcommand{\adn}{\downarrow}

\newcommand{\sbma}{\mbox{\scriptsize{\boldmath{$a$}}}}

\newcommand{\sbmb}{\mbox{\scriptsize{\boldmath{$b$}}}}

\newcommand{\sbmc}{\mbox{\scriptsize{\boldmath{$c$}}}}

\newcommand{\sbmd}{\mbox{\scriptsize{\boldmath{$d$}}}}

\newcommand{\sbme}{\mbox{\scriptsize{\boldmath{$e$}}}}

\newcommand{\sbmf}{\mbox{\scriptsize{\boldmath{$f$}}}}

\newcommand{\sbmg}{\mbox{\scriptsize{\boldmath{$g$}}}}

\newcommand{\sbmh}{\mbox{\scriptsize{\boldmath{$h$}}}}

\newcommand{\sbmi}{\mbox{\scriptsize{\boldmath{$i$}}}}

\newcommand{\sbmj}{\mbox{\scriptsize{\boldmath{$j$}}}}

\newcommand{\sbmk}{\mbox{\scriptsize{\boldmath{$k$}}}}

\newcommand{\sbml}{\mbox{\scriptsize{\boldmath{$l$}}}}

\newcommand{\sbmm}{\mbox{\scriptsize{\boldmath{$m$}}}}

\newcommand{\sbmn}{\mbox{\scriptsize{\boldmath{$n$}}}}

\newcommand{\sbmo}{\mbox{\scriptsize{\boldmath{$o$}}}}

\newcommand{\sbmp}{\mbox{\scriptsize{\boldmath{$p$}}}}

\newcommand{\sbmq}{\mbox{\scriptsize{\boldmath{$q$}}}}

\newcommand{\sbmr}{\mbox{\scriptsize{\boldmath{$r$}}}}

\newcommand{\sbms}{\mbox{\scriptsize{\boldmath{$s$}}}}

\newcommand{\sbmt}{\mbox{\scriptsize{\boldmath{$t$}}}}

\newcommand{\sbmu}{\mbox{\scriptsize{\boldmath{$u$}}}}

\newcommand{\sbmv}{\mbox{\scriptsize{\boldmath{$v$}}}}

\newcommand{\sbmx}{\mbox{\scriptsize{\boldmath{$x$}}}}

\newcommand{\sbmy}{\mbox{\scriptsize{\boldmath{$y$}}}}

\newcommand{\sbmz}{\mbox{\scriptsize{\boldmath{$z$}}}}

\newcommand{\sbmA}{\mbox{\scriptsize{\boldmath{$A$}}}}

\newcommand{\sbmB}{\mbox{\scriptsize{\boldmath{$B$}}}}

\newcommand{\sbmC}{\mbox{\scriptsize{\boldmath{$C$}}}}

\newcommand{\sbmD}{\mbox{\scriptsize{\boldmath{$D$}}}}

\newcommand{\sbmE}{\mbox{\scriptsize{\boldmath{$E$}}}}

\newcommand{\sbmF}{\mbox{\scriptsize{\boldmath{$F$}}}}

\newcommand{\sbmG}{\mbox{\scriptsize{\boldmath{$G$}}}}

\newcommand{\sbmH}{\mbox{\scriptsize{\boldmath{$H$}}}}

\newcommand{\sbmI}{\mbox{\scriptsize{\boldmath{$I$}}}}

\newcommand{\sbmJ}{\mbox{\scriptsize{\boldmath{$J$}}}}

\newcommand{\sbmK}{\mbox{\scriptsize{\boldmath{$K$}}}}

\newcommand{\sbmL}{\mbox{\scriptsize{\boldmath{$L$}}}}

\newcommand{\sbmM}{\mbox{\scriptsize{\boldmath{$M$}}}}

\newcommand{\sbmN}{\mbox{\scriptsize{\boldmath{$N$}}}}

\newcommand{\sbmO}{\mbox{\scriptsize{\boldmath{$O$}}}}

\newcommand{\sbmP}{\mbox{\scriptsize{\boldmath{$P$}}}}

\newcommand{\sbmQ}{\mbox{\scriptsize{\boldmath{$Q$}}}}

\newcommand{\sbmR}{\mbox{\scriptsize{\boldmath{$R$}}}}

\newcommand{\sbmS}{\mbox{\scriptsize{\boldmath{$S$}}}}

\newcommand{\sbmT}{\mbox{\scriptsize{\boldmath{$T$}}}}

\newcommand{\sbmU}{\mbox{\scriptsize{\boldmath{$U$}}}}

\newcommand{\sbmV}{\mbox{\scriptsize{\boldmath{$V$}}}}

\newcommand{\sbmW}{\mbox{\scriptsize{\boldmath{$W$}}}}

\newcommand{\sbmX}{\mbox{\scriptsize{\boldmath{$X$}}}}

\newcommand{\sbmY}{\mbox{\scriptsize{\boldmath{$Y$}}}}

\newcommand{\sbmZ}{\mbox{\scriptsize{\boldmath{$Z$}}}}

\newcommand{\bmnabla}{\mbox{\boldmath{$\nabla$}}}

\newcommand{\bmalpha}{\mbox{\boldmath{$\alpha$}}}

\newcommand{\bmone}{\mbox{\boldmath{$1$}}}

\newcommand{\bmsigma}{\mbox{\boldmath{$\sigma$}}}

\newcommand{\bmmu}{\mbox{\boldmath{$\mu$}}}

\newcommand{\bmell}{\mbox{\boldmath{$\ell$}}}

\newcommand{\bmGamma}{\mbox{\boldmath{$\Gamma$}}}

\newcommand{\La}{Leftarrow}

\newcommand{\raAB}{\overset{\ra}{AB}}

\newcommand{\raAA}{\overset{\ra}{AA}}

\newcommand{\lrda}{\longleftrightarrow \hspace*{-.25in}\longleftrightarrow}

 \newcommand{\lrra}{\longrightarrow \hspace*{-.21in}\longrightarrow}

\newcommand{\llla}{\longleftarrow \hspace*{-.21in}\longleftarrow}

\newcommand{\betallla}{\underset{\hspace*{-.25in}\beta}{\llla}}

\newcommand{\betaYlrra}{\underset{\hspace*{.2in}\beta Y}{\lrra}}

\newcommand{\betaYllla}{\underset{\hspace*{-.2in}\beta Y}{\llla}}

\newcommand{\lbetaetaYlrda}{\underset{\hspace*{-.2in}\beta \eta Y}{\lrda}}

\newcommand{\rbetaetaYlrda}{\underset{\hspace*{.2in}\beta \eta Y}{\lrda}}

\newcommand{\cbetaetaYlrda}{\underset{\beta \eta Y}{\lrda}}

\newcommand{\Ylrra}{\underset{\hspace*{.25in}Y}{\lrra}}

\newcommand{\rbetaetalrda}{\underset{\hspace*{.25in}\beta\eta}{\lrda}}

\newcommand{\rbetaYlrda}{\underset{\hspace*{.25in}\beta Y}{\lrda}}

\newcommand{\lbetaYlrda}{\underset{\hspace*{-.25in}\beta Y}{\lrda}}

\newcommand{\cbetaetalrda}{\underset{\beta\eta}{\lrda}}

\newcommand{\lbetaetalrda}{\underset{\hspace*{-.25in}\beta\eta}{\lrda}}

\newcommand{\Yllla}{\underset{\hspace*{-.25in}Y}{\llla}}

\newcommand{\betaetallla}{\underset{\hspace*{-.25in}\beta\eta}{\llla}}

\newcommand{\betaetalrra}{\underset{\hspace*{.25in}\beta\eta}{\lrra}}

\newcommand{\betaetaYlrra}{\underset{\hspace*{.15in}\beta \eta Y}{\lrra}}

\newcommand{\betaetaYllla}{\underset{\hspace*{-.15in}\beta \eta Y}{\llla}}

\newcommand{\betaetaYnlrda}{\underset{\!\!\!\!\!\beta\eta Y}{\not \!\!\!\! \! \lrda}}

\newcommand{\2}{\underline{2}}

\newcommand{\n}{\underline{n}}

\newcommand{\rat}{\rightarrowtail\!\!\!\!\rightarrow}

\newcommand{\lat}{\leftarrow\!\!\!\!\leftarrowtail}

\newcommand{\rra}{\rightarrow\!\!\!\!\rightarrow}

\newcommand{\uk}{\underline{k}}

\newcommand{\m}{\underline{m}}

\newcommand{\0}{\underline{0}}

\newcommand{\1}{\underline{1}}

\newcommand{\zk}{\underline{k-1}}

\newcommand{\xk}{\underline{k+1}}

\newcommand{\yP}{\underline{P}}

\newcommand{\boxk}{\framebox{\uk}}

\newcommand{\fzk}{\framebox{\zk}}

\newcommand{\fxk}{\framebox{xk}}

\newcommand{\boxm}{\framebox{\m}}

\newcommand{\boxn}{\framebox{\n}}

\newcommand{\rhu}{\rightharpoonup}

\newcommand{\orhua}{\overset{\rhu}{a}}

\newcommand{\orhub}{\overset{\rhu}{b}}

\newcommand{\orhui}{\overset{\rhu}{\imath}}

\newcommand{\orhuj}{\overset{\rhu}{\jmath}}

\newcommand{\orhuk}{\overset{\rhu}{k}}

\newcommand{\orhuF}{\overset{\rhu}{F}}

\newcommand{\orhuG}{\overset{\rhu}{G}}

\newcommand{\orhugj}{\overset{\rhu}{gj}}

\newcommand{\orhuT}{\overset{\rhu}{T}}

\newcommand{\orhuzero}{\overset{\rhu}{0}}

\newcommand{\orhunabla}{\overset{\rhu}{\nabla}}

\newcommand{\orhuu}{\overset{\rhu}{u}}

\newcommand{\orhuPQ}{\overset{\rhu}{PQ}}

\newcommand{\orhuOP}{\overset{\rhu}{OP}}

\newcommand{\orhuPR}{\overset{\rhu}{PR}}

\newcommand{\orhun}{\overset{\rhu}{n}}

\newcommand{\orhuv}{\overset{\rhu}{v}}

\newcommand{\orhur}{\overset{\rhu}{r}}

\newcommand{\orhuN}{\overset{\rhu}{N}}

\newcommand{\orhuS}{\overset{\rhu}{S}}

\newcommand{\orhubeta}{\overset{\rhu}{\beta}}

\newcommand{\orhud}{\overset{\rhu}{d}}

\newcommand{\orhudS}{\overset{\rhu}{dS}}

\newcommand{\orhuV}{\overset{\rhu}{V}}

\newcommand{\orhux}{\overset{\rhu}{x}}

\newcommand{\dint}{\displaystyle \int}

\newcommand{\diint}{\displaystyle \iint}

\newcommand{\diiint}{\displaystyle \iiint}

\newcommand{\teta}{\tilde{\eta}}

\newcommand{\ov}{\overline{v}}

\newcommand{\oDelta}{\overline{\Delta}}

\newcommand{\ovdelt}{\overline{\delta}}

\newcommand{\ox}{\overline{x}}

\newcommand{\oX}{\overline{X}}

\newcommand{\oV}{\overline{V}}

\newcommand{\os}{\overline{s}}

\newcommand{\oE}{\overline{E}}

\newcommand{\of}{\overline{f}}

\newcommand{\oD}{\overline{D}}

\newcommand{\dsum}{\displaystyle \sum}

\newcommand{\uszG}{\underset{z_1 \in {\cal G}_1}{\dsum}}

\newcommand{\uizG}{\underset{z_i \in {\cal G}_i}{\dsum}}

\newcommand{\uzG}{\underset{z \in {\cal G}}{\dsum}}

\newcommand{\uzGi}{\underset{z\in{\cal G}_i}{\dsum}}

\newcommand{\uj}{\underset{ij}{\dsum}}

\newcommand{\uijk}{\underset{ijk}{\dsum}}

\newcommand{\urhoG}{\underset{\rho_1 \in {\cal G}_1}{\dsum}}

\newcommand{\ui}{\underset{i}{\dsum}}

\newcommand{\obeta}{\overline{\beta}}

\newcommand{\ogamma}{\overline{\gamma}}

\newcommand{\odelta}{\overline{\delta}}

\newcommand{\ovarepsilon}{\overline{\varepsilon}}

\newcommand{\olambda}{\overline{\lambda}}

\newcommand{\uij}{\underset{ij}{\dsum}}

\newcommand{\R}{\mathbb R}

\newcommand{\G}{\mathbb G}

\newcommand{\Z}{\mathbb Z}

\newcommand{\omu}{\overline{\mu}}

\newcommand{\ovarep}{\overline{\varepsilon}}

\newcommand{\oA}{\overline{A}}

\newcommand{\oP}{\overline{P}}

\newcommand{\lra}{\longrightarrow}

\newcommand{\Sdiint}{\underset{S^{(n)}_{\varepsilon}\backslash

R}{\diint}}

\newcommand{\Rdiint}{\underset{R\backslash

S\nvar}{\diint}}

\newcommand{\var}{\varepsilon}

\newcommand{\nvar}{^{(n)}_{\varepsilon}}

\newcommand{\Crho}{C\|\rho(t)\|_{p_1}}

\newcommand{\cI}{{\cal I}}

\newcommand{\XV}{\left(\begin{array}{c}

X\\V\end{array}\right)}

\newcommand{\ooE}{\left( \begin{array}{c}

0\\0\\E\end{array}\right)}

\newcommand{\Xalpha}{\left(\begin{array}{l} X_{\alpha}\\P_{\alpha}

\end{array}\right)}

\newcommand{\aus}{\underset{\alpha}{\sum}}

\title{GLOBAL CLASSICAL SOLUTIONS OF THE ``ONE AND ONE-HALF" DIMENSIONAL VLASOV-MAXWELL-FOKKER-PLANCK SYSTEM\thanks{This work was supported by the National Science Foundation under the award DMS-1211667.}}

\author{Stephen Pankavich
\thanks {Department of Applied Mathematics and Statistics, Colorado School of Mines,
Golden, Colorado 80401, (pankavic@mines.edu)}
\and
Jack Schaeffer\thanks {Department of Mathematics Sciences, Carnegie Mellon University, Pittsburgh, PA  15213,
 (js5m@andrew.cmu.edu)}}
\maketitle

\begin{abstract}  We study the ``one and one-half" dimensional Vlasov-Maxwell-Fokker-Planck system and obtain the first results concerning well-posedness of solutions.  Specifically, we prove the global-in-time existence and uniqueness in the large of classical solutions to the Cauchy problem and a gain in regularity of the distribution function in its momentum argument.
\end{abstract}


\section{Introduction}

From a mathematical perspective, the fundamental non-relativistic equations which describe the time evolution of a collisionless plasma are given by the Vlasov-Maxwell system:

\vspace{.15in}

\noindent (VM) \hspace{.5in} $\left\{ \begin{array}{c} \partial_t f + v\cdot \nabla_x f + (E+v\times B) \cdot \nabla_v f = 0\\
\rho(t,x) = \dint\, f(t,x,v)\, dv,\ \ j(t,x) = \dint vf(t,x,v)\, dv\\
\partial_t E = \nabla \times B-j,\ \ \ \nabla \cdot E = \rho\\
\partial_t B= - \nabla \times E,\ \ \ \ \nabla \cdot B = 0.
\end{array} \right.$

\vspace{.15in}

\noindent Here, $f$ represents the distribution of (positively-charged) ions in the plasma, while $\rho$ and $j$ are the charge and current density, and $E$ and $B$ represent electric and magnetic fields generated by the charge and current.  The independent variables, $t \geq 0$ and $x,v \in \mathbb{R}^3$ represent time, position, and momentum, respectively, and physical constants, such as the charge and mass of particles, as well as the speed of light, have been normalized to one.

In order to include collisions of particles with a background medium in the physical formulation, a diffusive term is added to the Vlasov equation in (VM).  With this, the \change{equations are} referred to as the Vlasov-Maxwell-Fokker-Planck \change{system}.  Since basic questions of well-posedness remain unknown even in lower dimensions, we study a dimensionally-reduced version of this model for which $x \in \mathbb{R}$ and $v \in \mathbb{R}^2$, the so-called ``one and one-half dimensional" analogue, given by

\vspace{.15in}

\noindent (VMFP) \hspace{.25in}  $\left\{ \begin{array}{c} \partial_tf+v_1\partial_xf + K\cdot \nabla_v f = \Delta_v f\\
K_1 = E_1 + v_2B,\ \ \ K_2 = E_2 - v_1B\\
\rho(t,x) = \dint \, f(t,x,v)\,dv - \phi(x), \ \ \ j(t,x) = \dint \, vf(t,x,v)\,dv\\
\partial_tE_2 = -\partial_x B-j_2,\ \ \partial_t B = - \partial_xE_2,\ \ \partial_x E_1 = \rho,\ \ \partial_tE_1 = -j_1.\end{array}\right.$

\vspace{.15in}

\noindent This system is the lowest-dimensional analogue that one may study and include electromagnetic effects, as imposing $v \in \mathbb{R}$ changes the model into the one-dimensional Vlasov-Poisson system.  In (VMFP) we assume a single species of particles described by $f(t,x,v)$ in the presence of a given, fixed background $\phi \in C^1(\mathbb{R}) \cap H^1(\mathbb{R}) \cap L^1(\mathbb{R})$ that is neutralizing in the sense that

$$
\dint\phi (x)\, dx = \diint f(0,x,v) \, dv\,dx.
$$

\noindent The electric and magnetic fields are given by $E(t,x)=\langle E_1(t,x), E_2(t,x)\rangle$ and $B(t,x)$, respectively.  For initial data we take a nonnegative particle density $f^0$ with bounded moments $v^b_0\partial^k_x f^0 \in L^2 (\mathbb{R}^3)$, along with fields $E^0_2,B^0\in H^1(\mathbb{R})$.  Additionally, we specify data for $E_1$, namely

\vspace{.15in}

\noindent ($E_1DAT$) \hspace{.5in} $ E_1 (0,x) = \dint^x_{-\infty} \left(\dint f^0 (y,w)\, dw - \phi(y)\right)\, dy.$

\vspace{.15in}

\noindent In fact, this particular choice of data for $E_1$ is the only one which leads to a solution possessing finite energy (see \cite{5} and \cite{12}).  The inclusion of the neutralizing density $\phi$ is also necessary in order to arrive at finite energy solutions for (VMFP) with a single species of ion.

The analysis of (VM) has seen some progress in recent decades.  For instance, the global existence of weak solutions, which also holds for the relativistic system (RVM), was shown in \cite{3}.  Unlike its relativistic analogue, however, no results currently exist that ensure global existence of classical solutions.  Hence, the current work is focused in this direction.  Alternatively, a wide array of results have been obtained for the electrostatic simplification of (VM) -- the Vlasov-Poisson system, obtained by taking $B \equiv 0$ within the model.  The Vlasov-Poisson system does not include magnetic effects, and the electric field is given by an elliptic equation rather than a system of hyperbolic PDEs.  This simplification has led to a great deal of progress concerning the electrostatic system, including theorems regarding the well-posedness of solutions \cite{10, 11, 14, 15}.  The book \cite{6} can provide a general reference to information concerning kinetic equations of plasma dynamics, including (VM) and (VMFP).

Independent of these advances, many of the most basic existence and regularity questions remain unsolved for (VMFP). For much of the existence theory for collisionless models, one is mainly focused on bounding the velocity support of the distribution function $f$, assuming that $f^0$ possess\change{es} compact momentum support, as this condition has been shown to imply global existence \cite{7}.  Hence, one of the main difficulties which arises for (VMFP) is the introduction of particles that are propagated with arbitrarily large momenta, stemming from the inclusion of the diffusive Fokker-Planck operator.  Thus, the momentum support is necessarily unbounded and many known tools are unavailable.  Though the $v$-support of the distribution function is not bounded, we are able to overcome this issue by controlling large enough moments of the distribution to guarantee sufficient decay of $f$ in its momentum argument.  This also allows us to control nonlinear terms that arise within derivative estimates.  As an additional difference arising from the Fokker-Planck operator, we note that when studying collisionless systems, in which $\Delta_v f$ is omitted, $L^{\infty}$ is typically the proper space in which to estimate both the particle distribution and the fields.  With the addition of the diffusion operator, though, the natural space in which to estimate $f$ is now $L^2$.  Thus, to take advantage of the gain in regularity that should result from the Fokker-Planck term, we iterate in a weighted $L^2$ setting.  Other crucial features which appear include conservation of mass, and the symmetry of the diffusive operator.  The main advantage of the diffusion operator is that it allows one to estimate spatial derivatives of the density in $L^2(\mathbb{R}^3)$ independent of the momentum derivatives.  This is not true for the Vlasov-Maxwell system, which is conservative rather than dissipative.  Additionally, the \change{appearance of the} Laplacian allows the particle distribution to gain regularity in its momentum argument in comparison to its initial data. Finally, we note that our methods utilize an extra conservation law arising from the structure of the one-and-one-half dimensional system in order to bound the electric and magnetic fields.  Hence, they do not immediately apply to higher-dimensional analogues of (VMFP), though many of the other ideas presented below will likely be useful in the two, two-and-one-half, and three dimensional settings.

Though this is the first investigation of the well-posedness of (VMFP) in the large, others have studied Vlasov-Maxwell models incorporating a Fokker-Planck term for small initial data.  Both Yu and Yang \cite{17} and Chae \cite{1} constructed global classical solutions to the three-dimensional Vlasov-Maxwell-Fokker-Planck system for initial data sufficiently close to Maxwellian using Kawashima estimates and the well-known energy method.  Additionally, Lai \cite{8, 9} arrived at a similar result for a one and one-half dimensional ``relativistic" Vlasov-Maxwell-Fokker-Planck system using classical estimates.  The system in this work features a relativistic transport term, but still utilizes the Laplacian $\Delta_v$ as the Fokker-Planck term.  We note that the relativistic transport operator yields an extremely beneficial result, known as the cone estimate \change{(see \cite{5})}, whereas the non-relativistic transport within (VMFP) does not.  Thus, one essential component of the current paper is to overcome the lack of bounds on energy inside the light cone.  Finally, we mention \cite{12}, which arrived at similar results to our own but studied the relativistic Vlasov-Maxwell system with a Lorentz-invariant diffusion operator.  While we utilize some of the tools introduced within \cite{12}, and related articles \cite{4, 13}, we also introduce a number of new methods to overcome the loss of the cone estimate, finite speed of propagation, and {\it a priori} field bounds in order to arrive at the first large data global classical solutions to (VMFP) set in any dimension, see Theorem 1.2 below.  First we state a local existence theorem:

\begin{theorem} Let $a > 8$ and $\varepsilon > 0$ and denote

$$v_0 = \sqrt{1+|v|^2}.
$$

\noindent Assume that $\phi \in C^1(\mathbb{R}) \cap H^1(\mathbb{R}) \cap L^1(\mathbb{R})$. Assume that $f^0$ is continuous, nonnegative, and bounded and possesses a partial derivative with respect to $x$ such that

$$
\diint v^{a+2+\varepsilon}_0 (f^0)^2 \,dv\,dx + \diint v^{a-2+\varepsilon}_0 (\partial_xf^0)^2 \, dv\,dx
$$

\noindent is finite.  Assume that $E^0_2, B^0 \in C^1(\mathbb{R}) \cap H^1(\mathbb{R})$.  Then there is $T > 0$ depending only on

$$\diint \left[v^{a+2+\varepsilon}_0 (f^0)^2 + v^{a-2 + \varepsilon}_0 (\partial_x f^0)^2\right] \, dv\,dx + \| E^0_2\|^2_{H^1} + \|B^0\|^2_{H_1},$$

\noindent $f \in C([0,T] \times \mathbb{R}^3) \cap C^1((0,T] \times \mathbb{R}^3)$ with second order partial derivatives with respect to $v_1,v_2$ that are continuous on $(0,T] \times \mathbb{R}^3$, and $(E,B) \in C^1([0,T] \times \mathbb{R})$ for which (VMFP) holds, $(E_1DAT)$ holds, and

$$\left.(f,E_2,B)\right|_{t=0} = (f^0,E^0_2, B^0).
$$

\noindent Moreover, $f$ is nonnegative and bounded, and

$$
\diint \left[ v^{a+2 + \varepsilon}_0 f^2+v^{a-2+\varepsilon}_0 (\partial_x f)^2 \right] \, dv\,dx + \|E(t)\|_{H^1} + \|B(t)\|_{H^1}
$$

\noindent is bounded on $[0,T]$.  Lastly, the above solution is unique.
\end{theorem}

Note that $f^0$ is not assumed to be smooth in $v$.  Now we may state the main result:

\begin{theorem} In addition to the hypotheses of Theorem 1.1, assume that $E_2^0, B^0\in L^1(\mathbb{R})$ and $v^{\delta}_0 f^0 \in L^{\infty}(\mathbb{R}^3)$ for some $\delta > a + 2 + \varepsilon$, and $v^2_0f^0 \in L^1 (\mathbb{R}^3)$.  Then, the local solution of Theorem 1.1 may be extended to $[0,\infty) \times \mathbb{R}^3$.
\end{theorem}

We note that Theorems 1.1 and 1.2 can be altered to accommodate a friction term.  In the model with friction, the Vlasov equation is changed to

$$
\partial_t f + v_1\partial_x + k \cdot \nabla_v f = \nabla_v \cdot (\nabla_vf + vf).
$$

\noindent The new term is lower order and does not change \change{either of} the result\change{s}.

As additional evidence of the gain in regularity in $v$ we also state:

\begin{proposition} Assume the hypotheses of Theorem 1.2 hold.  Then for all $t > 0$

$$
\diint \left( f^2 + t \left| \nabla_v f \right|^2 + \dfrac{1}{2} t^2 \left|\nabla^2_v f\right|^2 \right) \, dv\,dx \leq C_t.
$$
\end{proposition}

This paper proceeds as follows.  The proof of Theorem 1.1 is postponed to Section 4 and Sections 2 and 3 assume the result of this theorem.  In Section 2 we state six lemmas and show how Theorem 1.2 follows from them.  The proofs of these lemmas and Proposition 1.3 are \change{contained within} Section 3.

Throughout the paper $C$ denotes a positive generic constant that may change from line to line.  When necessary, we will specifically identify the quantities upon which $C$ may depend.  Regarding norms, we will abuse notation and allow the reader to differentiate certain norms via context.  For instance $\|f(t)\|_{\infty} = \underset{x\in \mathbb{R},v\in \mathbb{R}^2}{\sup} \left| f(t,x,v)\right|$ whereas $\|B(t)\|_{\infty} = \underset{x \in \mathbb{R}}{\sup} \left| B(t,x)\right|$, with analogous statements for $\| \cdot \|_2$ and $< \cdot,\cdot>$ which denote the $L^2$ norm and inner product, respectively.

\section{Global Existence}

Throughout this section we assume the hypotheses of Theorem 1.1 hold.  Let $T$ be the maximal time of existence and, in order to prove Theorem 1.2 by contradiction, assume $T$ is finite.

To begin, we will first prove a result that will allow us to estimate the particle density and its moments.  When studying collisionless kinetic equations, one often wishes to integrate along the Vlasov characteristics in order to derive estimates.  However, the appearance of the Fokker-Planck term changes the structure of the operator in (VMFP), and the values of the distribution function are not conserved along such curves. Hence, the following lemma (similar to that of \cite{2}) will be utilized to estimate the particle distribution in such situations.

\begin{lemma} Let $g \in L^1((0,T),L^{\infty}(\mathbb{R}^3))$ and $h_0 \in L^{\infty}(\mathbb{R}^3) \cap L^2(\mathbb{R}^3)$  be given.  Let $F(t,x,v)=\mathcal{F}(t,x,v)+\mathcal{B}(t,x)\langle v_2,-v_1\rangle$ be given with $\mathcal{F} \in W^{1,\infty} ((0,T) \times \mathbb{R}^3;\mathbb{R}^2)$ and $\mathcal{B} \in W^{1,\infty} ((0,T) \times \mathbb{R};\mathbb{R})$.  Assume $h(t,x,v)$ is a weak solution of

\be
\left\{  \begin{array}{rcc}{\cal L}h & =&\partial_th +v_1 \partial_xh + F(t,x,v) \cdot \nabla_v h - \Delta_v h = g(t,x,v)\\
\\
&&  h(0,x,v) = h_0 (x,v)
\end{array}\right. \label{E2.1}
\ee

\noindent so that $h \in L^2((0,T) \times \mathbb{R}; H^1(\mathbb{R}^2))$ satisfies

$$\begin{gathered} \int_0^T \iint \biggl [ h \left ( -\partial_t \phi - v_1 \partial_x \phi  \right ) + \nabla_v h \cdot \left ( F \phi + \nabla_v \phi \right ) - g\phi \biggr ]  dv dx dt\\ - \iint h_0(x,v) \phi(0,x,v)  dv  dx = 0
\end{gathered}$$

\noindent for every $\phi \in \mathcal{D}([0,T) \times \mathbb{R}^3)$.
Then, for every $t \in [0,T]$

$$
\|h(t)\|_{\infty} \leq \|h_0 \|_{\infty} + \dint^t_0 \|g(s)\|_{\infty} \, ds.
$$
\end{lemma}

Another useful tool will be the conservation of mass and energy growth identities, which we establish in the next result.

\begin{lemma}{\rm (Conservation Laws).}  Assume $v^2_0 f^0 \in L^1 (\mathbb{R}^3)$.  Then, for every $t\in [0,T)$,

$$\|f(t)\|_1 = \|f^0\|_1
$$

\noindent and

$$
\diint |v|^2 f(t,x,v) \, dv\,dx + \dint(|E|^2 + B^2) dx \leq C(1+t).
$$
\end{lemma}

Next, we state a lemma that will allow us to control $v_2$ moments of the particle distribution.

\begin{lemma} {\rm (Propagation of $v_2$-moments).} Let $p \in [0,\infty)$ be given and assume the hypotheses of Lemma 2.2 with $E_2^0, B^0 \in L^1(\mathbb{R}).$  Let $R(s) = \sqrt{1+s^2}$.  If $\|R(v_2)^p f^0\|_{\infty} < \infty$, then for any $t \in [0,T)$

$$
\| R(v_2)^p f(t)\|_{\infty} < C_T.
$$
\end{lemma}

With control of velocities in the $v_2$ direction, we are able to control the induced electric and magnetic fields.  Bounds on moments of the particle density then follow from this result.

\begin{lemma} {\rm (Control of fields and moments).}  Assume there is $\delta > 4$ such that $v^{\delta}_0 f^0 \in L^{\infty}(\mathbb{R}^3), v^2_0 f^0 \in L^1(\mathbb{R}^3)$, and $B^0 \in L^1 (\mathbb{R})$.  Then, for any $t \in [0,T)$

\be
\|v^{\delta}_0 f(t)\|_{\infty} \leq C_T, \label{e2.2}
\ee

\be  \|E(t)\|_{\infty} + \|B(t)\|_{\infty} \leq C_T, \label{E2.3}
\ee

\noindent and

\be \left \| \dint v^{\beta - 2}_0 f(t)\, dv \right \|_{\infty} \leq C_T \label{E2.4}
\ee

\noindent for any $\beta \in [0,\delta)$.
\end{lemma}

Thus, once control of the fields is obtained, any higher moment of the particle distribution function can be controlled as well, assuming that the initial distribution possesses the same property.  Next, we utilize energy estimates to bound the density and its derivatives in $L^2 (\mathbb{R}^3)$.

\begin{lemma} Assume the hypotheses of Lemma 2.4 hold, then for every $t \in (0,T]$

$$\dfrac{d}{dt} \|f(t)\|^2_2 = -2 \|\nabla_v f(t)\|^2_2
$$

\noindent and thus

$$
\|f(t)\|_2 \leq \|f^0\|_2.
$$

\noindent If additionally, $v^{\gamma}_0 f^0 \in L^2 (\mathbb{R}^3)$ for some $\gamma > 0$, then

$$
\dfrac{d}{dt} \|v^{\gamma}_0 f(t)\|^2_2 \leq C_T \|v^{\gamma}_0 f(t)\|^2_2 - 2 \| v^{\gamma}_0 \nabla_vf(t)\|^2_2
$$

\noindent and thus

$$ \|v^{\gamma}_0 f(t) \|_2 \leq C_T$$

\noindent for every $t \in [0,T)$.
\end{lemma}

\begin{lemma} Assume the hypotheses of Lemma 2.4 hold with $\delta > 8$.  Then for every $\gamma \in \left( 2, \dfrac{\delta - 4}{2} \right) \cap \left( 2, \dfrac{a - 2 + \varepsilon}{2} \right]$ and $ t \in [0,T)$ we have

$$
\|v^{\gamma}_0 \partial_x f(t) \|_{L^2} + \| \partial_x E(t)\|_{L^2} + \| \partial_x B(t)\|_{L^2} \leq C_T.
$$
\end{lemma}

\begin{proof}  Now we may prove Theorem 1.2.  Applying Lemma 2.5 with $\gamma = \dfrac{a+2+\varepsilon}{2}$ yields

$$
\diint v^{a+2+\varepsilon}_0 f^2\, dv\,dx \leq C_T.
$$

\noindent Applying Lemma 2.6 with $\gamma = \dfrac{a-2+\varepsilon}{2} $ yields

$$
\diint v^{a-2+\varepsilon}_0 (\partial_x f)^2\, dv\,dx + \dint (|\partial_x E|^2 + (\partial_x B)^2)\, dx \leq C_T.
$$

\noindent Also by Lemma 2.2

$$\dint(|E|^2 +B^2)\,dx \leq C_T.
$$

\noindent Taking $(f(t),E_2(t), B(t))$ as an initial condition and applying Theorem 1.1 we find the solution may be extended to $[0,t+\change{\tau} ]$ with $\change{\tau} > C_T$.  This contradicts the maximality of $T$ and completes the proof.  \end{proof}

\section{Proofs of Lemmas and Estimates}  The first result (Lemma 2.1) is very close to a previous lemma \cite{12}, in which this property was shown for the relativistic Fokker-Planck operator.  One alteration necessary in the proof of [12, Lemma 1] is to change the relativistic velocity $\hat{v}_1$ to $v_1$, which does not affect the conclusion.  Also, here $F$ is not in $L^{\infty}$, but $\nabla_v\cdot F=\nabla_v\cdot\mathcal{F}$ and the proof of [12, Lemma 1] still applies.  Hence, we omit any additional details.

\begin{proof} {\rm [Lemma 2.2]}  We begin with conservation of mass.  Integrating the Vlasov equation over all $(x,v)$ we find

$$
\dfrac{d}{dt} \diint f(t,x,v)\, dv\,dx = 0.
$$

\noindent Thus, using the decay of $f^0$ we find for every $t \in [0,T)$

\be \diint f(t,x,v) \, dv dx = \diint f^0 (x,v) \, dv\, dx < \infty. \label{E3.1}
\ee

To arrive at the estimate of the total energy, we multiply the Vlasov equation by $|v|^2$ and integrate in $v$.  The Fokker-Planck term becomes

$$
\dint |v|^2 \Delta_v f\, dv = - \dint 2v \cdot \nabla_v f \, dv = 4 \dint f \, dv
$$

\noindent after two integrations by parts.  Hence, using the divergence structure of the Vlasov equation, we arrive at the local energy identity

\be
\partial_t e + \partial_x m = 4\dint f(t,x,v)\, dv \label{E3.2}
\ee

\noindent where

$$ e(t,x) = \dint |v|^2 f(t,x,v) \, dv + (|E(t,x)|^2 + |B(t,x)|^2)
$$

\noindent and

$$
m(t,x) = \dint v_1 |v|^2 f(t,x,v) dv + 2 E_2 (t,x) B(t,x).
$$

\noindent We integrate (\ref{E3.2}) over all space to deduce the global energy identity

$$\dfrac{d}{dt} \dint e(t,x) dx = 4 \diint f^0 (x,v) \, dx\, dv
$$

\noindent whence we find

$$
\dint e(t,x)\, dx \leq C(1+t)
$$

\noindent for all $t \in [0,T)$.  \end{proof}

Now we utilize the conservation laws to prove Lemma 2.3.

\begin{proof} {\rm [Lemma 2.3]} We begin by bounding the potential associated to the electric and magnetic fields.  By Lemma 2.2 we have

$$
\int \vert j_2(t,x) \vert \, dx \leq C(1+t)
$$

\noindent and hence

$$
\left \vert \int \int_0^t j_2(\tau, y\pm (t-\tau)) \, d \tau dy \right \vert \leq \int_0^t \int \vert j_2(\tau, y \pm (t - \tau)) \vert \, dy d\tau \leq C(1+t)^2.
$$

\noindent Since $B^0, E_2^0 \in L^1(\mathbb{R})$, it follows that

$$ \int \vert B(t,x) \vert \ dx \leq C(1+t)^2$$

\noindent and we may define

$$
A(t,x) = \dint^x_{-\infty} B(t,y) \, dy.
$$

\noindent Note that $\partial_x A = B$ and $\partial_t A = -E_2$.  Moreover, using Maxwell's equations, we find

$$
(\partial^2_t - \partial^2_x) A = j_2
$$

\noindent and thus

\be A(t,x) = \dfrac{1}{2} (A(0,x-t) + A(0, x+t)) + \dfrac{1}{2} \dint^t_0 \dint^{x+t-s}_{x-t+s}  \dint v_2 f(s,y,v) \, dv dy ds. \label{E3.3} \ee

\noindent The $(x,v)$-integral can be bounded using Cauchy-Schwarz and Lemma 2.2 as

$$\begin{array}{rcl} \dint^{x+t-s}_{x-t+s} \dint v_2 f(s,y,v) \, dv dy & \leq & \left( \diint f(x,y,v)\, dv dy\right)^{1/2} \left( \diint v^2_2 f(x,y,v) \, dv dy\right)^{1/2}\\
\\
& \leq & \|f^0 \|^{1/2}_1 \left( \diint |v|^2 f(s,y,v) \, dv dy \right)^{1/2}\\
\\
& \leq & C(1+s)^{1/2}.
\end{array}
$$

\noindent Hence, using the assumptions on initial data and integrating, we find

$$
\|A(t)\|_{\infty} \leq C (1+t)^{3/2} \leq C_T.
$$

Next, we utilize the identity

$$
\partial_t A + v_1 \partial_x A = -E_2 + v_1B = -K_2
$$

\noindent within the Vlasov-Fokker-Planck equation.  In particular, let $\psi \in C^2(\mathbb{R})$ be given and multiply this equation by $\psi (v_2 + A(t,x))$.  Denoting the VFP operator by

$$
{\cal V}h:= \partial_th + v_1 \partial_xh + K \cdot \nabla_vh - \Delta_v h,
$$

\noindent we find

\be {\cal V}(\psi(v_2 +A)f) = - 2 \psi' (v_2 + A) \partial_{v_2} f - \psi'' (v_2+A)f. \label{E3.4}
\ee

Next, define the function $R(x) = \sqrt{1+x^2}$.  To prove the first assertion, we use $\psi (x) = R^p(x)$ within (\ref{E3.4}) and derive the equation

$$
{\cal V}(R^p(v_2 + A)f) = -2 p(v_2 + A) R^{p-2} (v_2 + A) \partial_{v_2} f-p R^{p-4} (v_2+A)[1+(p-1)|v_2 + A|^2 ]f.
$$

\noindent Using the identity

$$
\partial_{v_2} f = R^{-p} (v_2 + A) \partial_{v_2} (R^p (v_2 + A) f) - p(v_2 +A)R^{-2} (v_2 + A)f
$$

\noindent the right side becomes

$$
-2p(v_2+A)R^{-2}(v_2+A)\partial_{v_2} (R^p(v_2 +A)f) + p R^{p-4} (v_2 + A) [ -1 + (p+1)|v_2 +A|^2]f.
$$

\noindent Hence, if this first term is included within the VFP operator by defining

$$ \overline{K} = K + \left\langle 0,2p \dfrac{v_2+A}{1+|v_2+A|^2} \right\rangle
$$

\noindent to form the new operator $\overline{\cal V}$, we find

$$
\overline{\cal V} (R^p(v_2+A)f) = pR^{p-4}(v_2+A) [-1 + (p+1)|v_2 + A|^2]f.
$$

\noindent We note that the term on the right side satisfies

$$
|pR^{p-4} (v_2 + A) [ -1 +(p+1) |v_2 + A|^2] f|\leq CR^{p-2} (v_2+A)f \leq CR^p(v_2+A) f.
$$

\noindent We invoke Lemma 2.1 with $h = R^p (v_2+A)f$ and ${\cal L} = \overline{\cal V}$ so that

$$
\|R^p(v_2+A(t))f (t) \|_{\infty} \leq \|R^p(v_2+A(0)) f^0 \|_{\infty} + C \dint^t_0 \| CR^p (v_2 + A(s))f(s)\|_{\infty}\, ds.
$$

\noindent By Gronwall's inequality we find

$$
\|R^p (v_2 + A(t))f (t) \|_{\infty} \leq C_T
$$

\noindent for $t \in [0,T)$.  Finally, the previously established control of $\|A(t)\|_{\infty}$ yields the first result as for $p \geq 0$

$$
\begin{array}{rcl} R^p(v_2)f(t,x,v) & = & (1+|v_2 + A(t,x) - A(t,x)|^2)^{p/2} f(t,x,v)\\
\\
& \leq & C (R^p(v_2+A(t,x)) + |A(t,x)|^p)f(t,x,v)\\
\\
& \leq & C\|R^p (v_2+A(t))f(t)\|_{\infty} + \| A(t)\|^p_{\infty} \|f(t)\|_{\infty}\\
\\
& \leq & C_T.
\end{array}
$$

\noindent Hence, taking supremums we find

$$
\|R^p(v_2)f(t)\|_{\infty} \leq C_T.
$$
\end{proof}

Using this result, we may bound the fields and moments of the distribution function.

\begin{proof} {\rm [Lemma 2.4]} We first bound $E_1$ using conservation of mass so that

$$\|E_1(t)\|_{\infty} = \underset{x \in \mathbb{R}}{\sup} \left| \dint^x_{-\infty} \left( \dint f(t,x,v)\, dv - \phi (x) \right) \, dx \right| \leq \diint f(t,x,v) \, dvdx + \|\phi\|_1 \leq C
$$

\noindent Next, we estimate the other field components.  Using the transported field equations, we find

\be
(E_2 \pm B)(t,x) = (E_2 \pm B)(0,x\mp t) - \dint^t_0 \dint v_2 f(s,x \mp (t-s),v)\, dvds. \label{E3.5}
\ee

\noindent Note that $E^0_2,B^0 \in L^{\infty}(\mathbb{R})$ by the Sobolev \change{embedding} theorem.   Thus, for any $\varepsilon_1, \varepsilon_2 > 0 $ we have

$$
\begin{array}{rcl} |(E_2 \pm B)(t,x)| & \leq & C\left(1 + \dint^t_0 \dint R^{-(1+\varepsilon_1)} (v_1) \left[ R^{1+\varepsilon_1}(v_1) f^{\frac{1+\varepsilon_1}{\gamma}} (s,x \mp (t-s), v)\right]\right.\\
\\
& & \left. |v_2| R^{-(2+\varepsilon_2)} (v_2) \left[ R^{2+\varepsilon_2} (v_2) f^{1-\frac{1+\varepsilon_1}{\gamma}} (s,x \mp (t-s),v)\right]\, dvds \right)\\
\\
& \leq & C \left( 1+ \dint^t_0 \|R^{\gamma} (v_1) f(s)\|^{\frac{1+\varepsilon_1}{\gamma}}_{\infty} \| R^q (v_2) f(s) \|^{\frac{\gamma -(1+\varepsilon_1)}{\gamma}}_{\infty} \, ds \right)
\end{array}
$$

\noindent where $q = \dfrac{(2+\varepsilon_2)\gamma}{\gamma - (1+\varepsilon_1)}$.  We choose $\gamma > 1 + \varepsilon_1$ and note that $\delta > 4$ ensures that we may also choose $q \leq \gamma \leq \delta$.  Define the function

$$
F(t) := \underset{s \in[0,t]}{\sup} \|v^{\gamma}_0 f(s)\|_{\infty}.
$$

\noindent Invoking Lemma 2.3 with $p = q$ we find

\be\begin{array}{rcl} \|(E_2 \pm B)(t)\|_{\infty}& \leq & C_T \left( 1 + \left[ \underset{s \in [0,t]}{\sup} \| R^{\gamma}(v_1) f(s)\|_{\infty} \right]^{\frac{1+\varepsilon_1}{\gamma}} \right)\\
\\
& \leq & C_T \left(1+F(t)^{\frac{1+\varepsilon_1}{\gamma}}\right). \end{array}
\label{E3.6}
\ee

\noindent Using the identity

$$
E_2(t,x) = \dfrac{1}{2} ([E_2(t,x) + B(t,x)] + [E_2(t,x)-B(t,x)])
$$

\noindent we see that the same bound holds for $\|E_2(t)\|_{\infty}.$

Next, we multiply VFP by $v^{\gamma}_0$ and use the same method as in the proof of Lemma 2.3 to derive the equation

\be
{\cal V}(v^{\gamma}_0 f) = \gamma v^{\gamma - 2}_0 (v \cdot E) f-2\gamma v^{-2}_0 v \cdot \nabla_v (v^{\gamma}_0 f) + \gamma(\gamma|v|^2 -2)v^{\gamma -4}_0 f. \label{E3.7}
\ee

\noindent If the second term on the right side is included within the VFP operator by defining

$$
\overline{K} = K + 2 \gamma \dfrac{v}{1+|v|^2}
$$

\noindent to form the new operator $\overline{\cal V}$, we find

$$
\begin{array}{rcl}
\overline{\cal V} (v^{\gamma}_0 f)& = & \gamma v^{\gamma -2}_0 (v \cdot E) f + \gamma(\gamma|v|^2-2) v^{\gamma -4}_0 f\\
\\
& =: & I + II. \end{array}
$$

\noindent Clearly,

$$
II \leq C \| v^{\gamma - 2}_0 f(t)\|_{\infty} \leq C \|v^{\gamma}_0 f(t) \|_{\infty} \leq CF(t).
$$

\noindent Estimating $I$ requires the field estimates, which yield

$$
\begin{array}{rcl}
I & \leq & C v^{\gamma -2}_0 (|v_1| + |v_2| \cdot \|E_2(t)\|_{\infty})f\\
\\
& \leq & C \left( \|v^{\gamma -1}_0 f(t) \|_{\infty} + C_T \||v_2|^{\frac{\gamma}{2}} f(t)\|^{\frac{2}{\gamma}}_{\infty}
\|v^{\gamma}_0 f(t) \|^{1-\frac{2}{\gamma}}_{\infty} \left(1 + F(t)^{\frac{1+\varepsilon_1}{\gamma}} \right) \right) \\
\\
& \leq & C_T \left( F(t) + F(t)^{\frac{\gamma - 2}{\gamma}} + F(t)^{\frac{\gamma -1 + \varepsilon_1}{\gamma}}\right)
\end{array}
$$

\noindent since $ \gamma/2 \leq \delta$.  We combine these estimates and invoke Lemma 2.1 with $h = v^{\gamma}_0 f$ and ${\cal L} = \overline{\cal V}$ so that

$$
\|v^{\gamma}_0 f(t) \|_{\infty} \leq \|v^{\gamma}_0f^0\|_{\infty} + C_T \dint^t_0 \left( F(t) + F(t)^{\frac{\gamma-2}{\gamma}} + F(s)^{\frac{\gamma-1+\varepsilon_1}{\gamma}} \, ds \right).
$$

\noindent Taking the supremum in $t$ and choosing $\varepsilon_1 \leq 1$

$$
F(t) \leq F(0) + C_T \dint^t_0 \left( F(s) + F(s)^{\frac{\gamma - 2}{\gamma}} + F(s)^{\frac{\gamma -1+\varepsilon_1}{\gamma}}\right) \, ds \leq F(0) + C_T \dint^t_0 (1+F(s))\, ds.
$$

\noindent Gronwall's inequality then yields the bound $F(t) \leq C_T$ for any $t \in [0,T)$ and $1 \leq \gamma \leq \delta$. The bound on moments of the distribution function follows immediately and the field bound

$$\|E_2(t) \|_{\infty} + \|B(t)\|_{\infty} \leq C_T
$$

\noindent then follows from (\ref{E3.6}).  Finally, using the bound on moments of the density, control of the $v$-integral follows since we have

$$
\dint v^{\gamma -2}_0 f(t,x,v)\, dv \leq \| v^{\delta}_0 f(t) \|_{\infty} \cdot \dint v^{-2-(\delta - \gamma)}_0 \, dv \leq C_T
$$

\noindent for $\gamma < \delta$ and taking the supremum in $x$ yields (\ref{E2.4}).
\end{proof}

\begin{proof} {\rm [Lemma 2.5]} We proceed by using dissipative estimates.  First, we compute:

$$
\begin{array}{rcl}
\dfrac{1}{2} \dfrac{d}{dt} \|f(t)\|^2_2 & = & \langle -v_1 \partial_x f - K \cdot \nabla_v f + \Delta_v f,f\rangle\\
\\
& = & - \langle v_1 \partial_x f,f \rangle - \langle K \cdot \nabla_v f,f \rangle + \langle \Delta_v f,f\rangle.
\end{array}
$$

\noindent Notice that the first two terms are pure derivatives in $x$ and $v$, respectively.  Thus,

$$
\langle v_1 \partial_x f,f \rangle = \dfrac{1}{2} \diint \partial_x (v_1 f^2) \, dvdx = 0
$$

\noindent and

$$
\langle K \cdot \nabla_v f,f\rangle = \dfrac{1}{2} \diint \nabla_v \cdot (Kf^2) \, dv dx = 0.
$$

\noindent Finally $\langle \Delta_v f,f\rangle = - \|\nabla_v f(t) \|^2_2$.  Hence

$$
\dfrac{d}{dt} \|f(t)\|^2_2 = - 2 \|\nabla_v f(t)\|^2_2 \leq 0
$$

\noindent and the first conclusion follows.

Similarly, we may multiply by $v^{2\gamma}_0$ and proceed in the same manner.  Within this estimate we will use $v_0 \geq 1$ in order to increase moments of the estimates where necessary so as to match the results of the lemma.  Computing the time derivative

$$
\begin{array}{rcl}  \dfrac{1}{2} \dfrac{d}{dt} \| v^{\gamma}_0 f(t) \|^2_2 & = & \diint v^{2\gamma}_0 f[-v_1 \partial_x f - K \cdot \nabla_v f + \Delta_v f ] \, dvdx\\
\\
& = & I + II + III.
\end{array}
$$

\noindent The first term vanishes as it is a pure $x$-derivative.  For $II$, we integrate by parts and use the field bounds of Lemma 2.4 so that

$$
\begin{array}{rcl}
II & = & -\dfrac{1}{2} \diint v^{2\gamma}_0 \nabla_v \cdot (Kf^2) \, dvdx\\
\\
& = & \gamma \diint v^{2\gamma-2}_0 v \cdot Ef^2\, dvdx\\
\\
& \leq & C_T \|v^{\gamma}_0 f(t)\|^2_2.
\end{array}
$$

\noindent To estimate $III$, we integrate by parts twice in the first term and once in the second term to find

$$
\begin{array}{rcl}
III & = & -\diint \nabla_v (v^{2\gamma}_0 f) \cdot \nabla_v f\, dvdx\\
\\
& = & - \diint (2\gamma v^{2\gamma-2 }_0 v f + v^{2 \gamma}_0 \nabla_v f ) \cdot \nabla_v f\, dvdx\\
\\
& \leq & C \|v^{\gamma}_0 f(t) \|^2_2 - \|v^{\gamma}_0 \nabla_v f(t) \|^2_2.
\end{array}
$$

\noindent Combining the estimates, we find

$$
\dfrac{d}{dt} \|v^{\gamma}_0 f(t) \|^2_2 \leq C_T \|v^{\gamma}_0f(t)\|^2_2 - 2\|v^{\gamma}_0 \nabla_v f(t)\|^2_2
$$

\noindent as in the statement of the lemma.  Additionally, because the second term on the right side of the inequality is nonpositive, we invoke Gronwall's inequality and find

$$
\| v^{\gamma}_0 f(t) \|^2_2 \leq C_T\|v^{\gamma}_0 f^0\|^2_2 \leq C_T
$$

\noindent for every $\gamma \geq 0$ for which the norm of the initial data $\|v^{\gamma}_0 f^0\|_2 $ is finite.
\end{proof}

\begin{proof} {\rm [Lemma 2.6]} If $f$ were $C^3$ we could compute

$$
\begin{array}{rcl}
\dfrac{d}{dt} \diint v^{2\gamma}_0 (\partial_x f)^2 \, dvdx & = & \diint 2v^{2\gamma}_0 \partial_x f(\Delta_v \partial_x f-v_1\partial^2_x f\\
\\
& & - K \cdot \nabla_v \partial_x f - \partial_x K \cdot \nabla_v f)\, dvdx\\
\\
& = & - 2 \diint v^{2\gamma}_0 |\nabla_v\partial_x f|^2 \, dvdx + \diint (\partial_x f)^2 (\Delta_v v^{2\gamma}_0 + K \cdot \nabla_v v^{2\gamma}_0 ) \, dvdx\\
\\
&& + 2 \diint f\partial_x K \cdot (v^{2\gamma}_0 \nabla_v \partial_x f+ \partial_x f \nabla_v v^{2\gamma}_0 )\, dv dx\\
\\
& \leq & -2 \diint v^{2\gamma}_0 |\nabla_v\partial_x f|^2 \, dvdx + C \diint (\partial_x f)^2 v^{2\gamma}_0 (1 + |E|) \, dvdx \\
\\
& & + C \sqrt{\diint f^2 |\partial_x K|^2 v^{2\gamma}_0 \, dv dx} \left( \sqrt{\diint v^{2\gamma}_0 |\nabla_v \partial_x f |^2 \, dvdx}\right.\\
 \\
 &&\left. + \sqrt{\diint v^{2\gamma}_0 (\partial_x f)^2 \, dvdx }\right).
\end{array}
$$

\noindent Using the inequalities $-x^2 + A x \leq \dfrac{1}{4}A^2$ and $2xy \leq x^2 + y^2$ yields

$$
\begin{array}{rcl}
\dfrac{d}{dt} \diint v^{2\gamma}_0 (\partial_x f)^2 \, dvdx & \leq & C\diint (\partial_x f)^2 v^{2\gamma}_0  (1+|E|) \, dvdx \\
\\
& & + C\diint f^2 |\partial_x K|^2 v^{2\gamma}_0 \, dvdx
\end{array}
$$

\noindent and hence

\be\begin{array}{rcl}  \diint v^{2\gamma}_0 (\partial_x f)^2 \, dvdx & \leq &C + C \dint^t_0 \diint v^{2\gamma}_0 (\partial_x f)^2 ( 1+|E|) \, dv dx d\tau\\
\\
& & + C \dint^t_0 \diint v^{2\gamma}_0 f^2 |\partial_x K|^2 \, dv dx d\tau.
\end{array}
\label{E3.8}
\ee

\noindent By a \change{standard} regularization argument it follows that (\ref{E3.8}) holds for the solution $(f,E,B)$ with the regularity stated in Theorem 1.1.  Applying (\ref{E2.3}) and (\ref{E2.4}) with $\beta = 2 \gamma + 4$ yields

\be
\begin{array}{rcl}
\diint v^{2\gamma}_0 (\partial_xf)^2 \, dvdx & \leq & C + C_T \dint^t_0 \diint v^{2\gamma}_0 (\partial_x f)^2 \, dv dx d\tau \\
\\
& & + C \dint^t_0 \dint (|\partial_x E|^2 + (\partial_x B)^2) \|f(t)\|_{L^{\infty}} \dint f v^{2\gamma + 2}_0 \, dv dx d\tau\\
\\
& \leq & C + C_T \dint^t_0 \diint v^{2 \gamma}_0 (\partial_x f)^2 \, dv dx d\tau\\
\\
& & + C_T \dint^t_0 \dint (|\partial_x E|^2 + (\partial_x B)^2) \, dx d\tau.
\end{array} \label{E3.9}
\ee

\noindent \change{Similarly,} if $E$ and $B$ were $C^2$ we could compute

$$
\begin{array}{rcl} \dfrac{d}{dt} \dint [ | \partial_x E|^2 + (\partial_x B)^2] \, dx & = & -2 \dint \partial_x E \cdot \partial_x j \, dx\\
\\
& \leq & 2 \sqrt{\dint |\partial_x E|^2 \, dx}\ \sqrt{\dint |\partial_x j|^2 \, dx} \end{array}
$$

\noindent so

\be
\begin{array}{rcl} \dint [|\partial_x E|^2 + (\partial_x B)^2 ]\, dx & \leq & C+ 2 \dint^t_0 \sqrt{\dint|\partial_x E|^2 \, dx} \ \sqrt{\dint|\partial_x j |^2 \, dx} \, d\tau \\
\\
  & \leq & C+\dint^t_0 \dint (|\partial_x E|^2 + |\partial_x j |^2 ) \, dx d \tau.
  \end{array} \label{E3.10}
  \ee

\noindent By a \change{standard} regularization argument it follows that (\ref{E3.10}) holds for $E$ and $B \in C^1$.  Since $\gamma > 2$ we have

$$
\begin{array}{rcl}
|\partial_x j|^2 & \leq & \left( \dint |\partial_x f|v_0 dv\right)^2\\
\\
& \leq & \dint |\partial_x f|^2 v^{2\gamma}_0 dv \dint v^{2-2\gamma}_0 dv \leq C \dint |\partial_x f|^2 v^{2 \gamma}_0\, dv
\end{array}
$$

\noindent so adding (\ref{E3.9}) and (\ref{E3.10}) yields

$$
\begin{array}{rl} & \diint v^{2\gamma}_0 (\partial_xf)^2 \, dvdx + \dint (|\partial_x E|^2 + (\partial_x B)^2 ) \, dx\\
\\
\leq & C + C_T \dint^t_0 \left( \dint |\partial_x E|^2 \, dx d\tau + \diint v^{2\gamma}_0 (\partial _x f)^2 \, dv dx \right)\, d\tau.
\end{array}
$$

\noindent An application of Gronwall's inequality completes the proof.
\end{proof}

\begin{proof} {\rm [Proposition 1.3]} If $f$ were $C^4$ we could compute the following:

$$\begin{array}{rcl}
\dfrac{d}{dt} \diint f^2\, dvdx & = & -2 \diint |\nabla_v f|^2\, dvdx,\\
\\
\dfrac{d}{dt} \diint |\nabla_v f|^2 \, dvdx & = & - 2 \diint (|\nabla^2_v f|^2 + \partial_{v_1} f \partial_x f)\, dv dx,\\
\\
\dfrac{d}{dt} \diint | \nabla^2_v f|^2 \, dvdx & = & - 2 \diint |\nabla^3_v f|^2 \, dvdx - 4 \diint \nabla_v \partial_{v_1} f \cdot \nabla_v \partial_xf\, dvdx\\
\\
& = & -2 \diint |\nabla^3_v f|^2 \, dvdx + 4 \diint \partial_x f \Delta_v \partial_{v_1} f\, dvdx,
\end{array}
$$

\noindent so

$$\begin{array}{rl}
& \dfrac{d}{dt} \diint (f^2 + t |\nabla_v f|^2 + \dfrac{1}{2} t^2 |\nabla^2_v f|^2) \, dv dx \\
\\
=  & - \diint |\nabla_v f|^2 \, dv dx - t \diint |\nabla^2_v f|^2 \, dv dx - 2 t \diint \partial_{v_1} f \partial_x f \, dv dx \\
\\
& + \dfrac{1}{2} t^2 \diint (-2|\nabla^3_v f|^2 + 4\partial_x f \Delta_v \partial_{v_1} f) \, dv dx\\
\\
\leq & - \diint |\nabla_v f|^2 \, dv dx + 2 t \sqrt{\diint (\partial_x f)^2 \, dvdx} \ \sqrt{ \diint (\partial_{v_1} f)^2\, dv dx}\\
\\
& + t^2 \left( - \diint |\nabla^3_v f|^2 \, dvdx + 2 \sqrt{\diint (\partial_xf)^2 \, dv dx}\ \sqrt{\diint (\Delta_v \partial_{v_1} f)^2 \, dvdx}\right).
\end{array}
$$

\noindent Using the inequality $-x^2 + A x \leq \dfrac{1}{4} A^2$ twice yields

$$
\begin{array}{rl} & \dfrac{d}{dt} \diint (f^2 + t |\nabla_v f|^2 + \dfrac{1}{2} t^2 |\nabla^2_v f|^2 ) \, dvdx \\
\\
\leq & Ct^2 \diint (\partial_x f)^2 \, dv dx
\end{array}
$$

\noindent and

\be \begin{array}{rl}
& \diint (f^2 + t |\nabla_v f|^2 + \dfrac{1}{2} t^2 |\nabla^2_v f|^2) \, dv dx\\
\\
\leq & \diint (f^0)^2 \, dvdx + C \dint^t_0 \tau ^2 \diint (\partial_x f)^2 \, dv dx d\tau. \end{array} \label{E3.11}
\ee

\noindent \change{Again, by a standard} regularization argument it follows that (\ref{E3.11}) holds for the solution constructed in Theorem 1.1.  By Lemma 2.6

$$
\diint (\partial_x f)^2 \, dvdx \leq C_t
$$

\noindent so the proposition follows from (\ref{E3.11}).
\end{proof}

\section{Local Existence}

Define $b = a-4$,

$$
{\cal F} = \left\{ f: \mathbb{R}^3 \ra [0, \infty) \biggr \vert \  v^{\frac{a}{2}}_0 f, \ v^{\frac{b}{2}}
 \partial_x f \in L^2 (\mathbb{R}^3)\right\}
 $$

\noindent and

$$
\| f\|_{\cal F} = \| v^{\frac{a}{2}}_0 f\|_{L^2} + \| v^{\frac{b}{2}}_0 \partial_x f \|_{L^2}.
$$

\noindent For the time being we consider smooth  initial data $(f^0, E^0_2, B^0)$.

Let $T\in (0,1), R> 1, \psi: \mathbb{R} \ra [0,1]$ be smooth with $s \leq -1 \Ra \psi(s) = 1$ and $s \geq 0 \Ra \psi(s) =0$.  Define $\psi^R (v) = \psi (|v|-R)$.  For $(E,B) \in C([0,T];\ H^1(\mathbb{R}))$ smooth define

$$
{\cal L}^R (E,B) = ( \tilde{f}, \tilde{E}, \tilde{B})
$$

\noindent by

\be K = E + \psi^R(v) B(v_2, -v_1), \label{E4.1}
\ee

\be \partial_t \tilde{f} + v_1 \partial_x \tilde{f} + K \cdot \nabla_v \tilde{f} = \Delta_v \tilde{f}, \qquad \tilde{f} (0, \cdot, \cdot) = f^0,
\label{E4.2}
\ee

\be \tilde{\rho} = \dint \tilde{f} dv - \phi , \ \tilde{j} = \dint v \tilde{f} \, dv,
\label{E4.3}
\ee

\be \tilde{E}_1 = \dint^x_{-\infty} \tilde{\rho}\, dy,
\label{E4.4}
\ee

\be \partial_t \tilde{E}_2 + \partial_x \tilde{B} = - \tilde{j}_2, \qquad  \partial_t \tilde{B} + \partial_x \tilde{E}_2 = 0,
\label{E4.5}
\ee

\be (\tilde{E}_2, \tilde{B}) (0, \cdot) = (E^0_2, B^0).\label{E4.6}
\ee

\noindent Note that $K$ is bounded and hence by Proposition A.1 of \cite{2}, (\ref{E4.2}) has a solution
$\tilde{f} \in L^2([0,T] \times \mathbb{R}; H^1(\mathbb{R}^2))$.  Reference \cite{16} may be used for this also.

Let $\alpha = a + 2 + \varepsilon,\ \beta = b + 2 + \varepsilon$, and

$$
C_0 > \|E^0_2\|_{H^1} + \|B^0\|_{H^1} + \|v^{\frac{\alpha}{2}}_0 f^0 \|_{L^2} + \| v^{\frac{\beta}{2}}_0 \partial_x f^0 \|_{L^2}.
$$

\noindent We assume that

\be \|(E,B)(t)\|_{H^1}\leq 10C_0 \label{E4.7}
\ee

\noindent on $[0,T]$. \change{Within the remainder of this section} constants may depend on $\alpha, T$ and $C_0$ but not on $R$ or $\nabla_v f^0$.

First, using (\ref{E4.7}) and the Sobolev \change{embedding} theorem,

\be \begin{array}{rcl} \dfrac{d}{dt} \diint v^{\alpha}_0 \tilde{f}^2 \, dvdx & = & -2 \diint v^{\alpha}_0 | \nabla_v \tilde{f}|^2\, dv dx\\
\\
& & + \diint \tilde{f}^2 (\Delta_v v^{\alpha}_0 + E \cdot \nabla_v v^{\alpha}_0 ) \, dv dx \\
\\
& \leq & C \diint \tilde{f}^2 v^{\alpha}_0 \, dv dx. \end{array}
 \label{E4.8}
\ee

\noindent Hence, by Gronwall's inequality

\be \diint v^{\alpha}_0 \tilde{f}^2 \, dv dx \leq C.\label{E4.9}
\ee

\noindent Similarly, and using the Cauchy Schwartz inequality

\be
\begin{array}{rcl}
\dfrac{d}{dt} \diint v^{\beta}_0 (\partial_x \tilde{f})^2 \, dv dx & = &-2 \diint v^{\beta}_0 | \nabla_v \partial_x \tilde{f}|^2 \, dvdx \\
\\
& & + \diint (\partial_x\tilde{f})^2 (\Delta_v v^{\beta}_0 + E \cdot \nabla_v v^{\beta}_0) \, dvdx \\
\\
& & + 2 \diint \tilde{f} \partial_x K \cdot ( v^{\beta}_0 \nabla_v \partial_x \tilde{f} + \partial_x \tilde{f} \nabla_v v^{\beta}_0 )\, dvdx\\
\\
& \leq & - 2 \diint v^{\beta}_0 |\nabla_v \partial_x \tilde{f}|^2 \, dv dx + C\diint (\partial_x \tilde{f})^2 v^{\beta}_0 \, dv dx \\
\\
& & + C \diint \tilde{f} (|\partial_x E| + | \partial_x B|) ( v^{\beta + 1}_0 |\nabla_v \partial_x \tilde{f}| + | \partial_x \tilde{f}| v^{\beta}_0 ) \, dv dx \\
\\
& \leq & - 2 \diint v^{\beta}_0 |\nabla_v \partial_x \tilde{f}|^2 \, dvdx + C \diint (\partial_x \tilde{f})^2 v^{\beta}_0 \, dv dx \\
\\
& & + C \sqrt{\diint \tilde{f}^2 v^{\beta + 2}_0(|\partial_x E|^2 + (\partial_x B)^2) \, dv dx}  \left[ \sqrt{\diint v^{\beta}_0 |\nabla_v \partial_x \tilde{f}|^2 \, dvdx } \right.\\
\\
& &\left.  + \sqrt{\diint v^{\beta}_0 (\partial_x \tilde{f})^2\, dv dx}\ \right].
\end{array} \label{E4.10}
\ee

\noindent Using $-x^2 + A x \leq \dfrac{1}{4} A^2$ and $2xy \leq x^2 + y^2$ yields

\be \begin{array}{rcl} \dfrac{d}{dt} \diint v^{\beta}_0 (\partial_x \tilde{f})^2\, dvdx & \leq  & C\diint v^{\beta}_0 (\partial_x \tilde{f})^2 \, dv dx \\
\\
&& + C \diint \tilde{f}^2 v^{\beta +2}_0 dv (|\partial_x E|^2 + (\partial_x B)^2)\, dx. \end{array}
\label{E4.11}
\ee

\noindent Also by (\ref{E4.9})

\be \begin{array}{rcl}  \dint \tilde{f}^2 v^{\beta +2}_0 dv & = & \dint \tilde{f}^2 v^{\frac{\alpha + \beta}{2}}_0 dv = \dint^x_{-\infty} \dint 2\tilde{f} \partial_x \tilde{f} v^{\frac{\alpha + \beta}{2}}_0 \, dvdy\\
\\
& \leq & 2 \|\tilde{f}(t) v^{\frac{\alpha}{2}}_0 \|_{L^2} \| \partial_x \tilde{f} (t) v^{\frac{\beta}{2}} \|_{L^2}\\
\\
& \leq & C + \diint v^{\beta}_0 (\partial_x \tilde{f})^2\, dvdx
\end{array}
\label{E4.12}
\ee

\noindent so using (\ref{E4.7}), (\ref{E4.11}) yields

$$
\dfrac{d}{dt} \diint v^{\beta}_0 (\partial_x \tilde{f})^2\, dvdx \leq C + C\diint v^{\beta}_0 (\partial_x \tilde{f})^2\, dvdx.
$$

\noindent Hence

\be
\diint v^{\beta}_0 (\partial_x \tilde{f})^2 \, dvdx \leq C. \label{E4.13}
\ee

Next consider $({\cal E},{\cal B}) \in C([0,T]; H^1(\mathbb{R}))$ smooth for which (\ref{E4.7}) holds and define $(\tilde{F}, \tilde{\cal E}, \tilde{\cal B}) = {\cal L}^{\cal R} ({\cal E}, {\cal B})$ where $R \leq {\cal R}$ and ${\cal K}, \tilde{F}, \tilde{P}, \tilde{J}, \tilde{\cal E},$ and $\tilde{\cal B}$ are defined as in equations (\ref{E4.1}) through (\ref{E4.6}).

Let $G = (E,B) - ({\cal E},{\cal B})$ and note that

\be
\begin{array}{rcl}
|K - {\cal K}| & \leq & v_0 |G| + v_0 |{\cal B}|\ |\psi^R(v) - \psi^{\cal R} (v)|\\
\\
& \leq & v_0 |G| + v^{1+\frac{\varepsilon}{2}} |{\cal B}| R^{-\frac{\varepsilon}{2}}
\end{array}
\label{E4.14}
\ee

\noindent and similarly

\be |\partial_x K - \partial_x {\cal K}| \leq v_0 | \partial_x G| + v^{1+\frac{\varepsilon}{2}}_0 | \partial_x {\cal B} | R^{-\frac{\varepsilon}{2}}. \label{E4.15}
\ee

Let $\tilde{g} = \tilde{f} - \tilde{F}$.  Proceeding as before in (\ref{E4.10}) and (\ref{E4.11}) we have

\be
\begin{array}{rcl} \dfrac{d}{dt} \diint v^{a}_0 \tilde{g}^2 \, dvdx & = & -2 \diint v^a_0 |\nabla_v \tilde{g}|^2 \, dvdx \\
\\
&& + \diint \tilde{g}^2 (\Delta_v v^a_0 + E \cdot \nabla_v v^a_0)\, dvdx \\
\\
&& + 2 \diint \tilde{F} (K-{\cal K}) \cdot (v^a_0 \nabla_v \tilde{g} + \tilde{g} \nabla_v v^a_0 ) \, dvdx\\
\\
& \leq & C \diint \tilde{g}^2 v^a_0 \, dvdx + C \diint \tilde{F}^2 |K-{\cal K}|^2 v^a_0 \, dvdx.
\end{array}
\label{E4.16}
\ee

\noindent By (\ref{E4.14}), the Sobolev \change{embedding} theorem, and (\ref{E4.9}) we have

$$
\begin{array}{rl} & \diint \tilde{F}^2|K-{\cal K}|^2 v^a_0 \, dvdx \leq \diint \tilde{F}^2 (G^2+{\cal B}^2
R^{-\varepsilon})v^{\alpha}_0 \, dvdx\\
\\
\leq & C(\|G(t)\|^2_{L^{\infty}} + \| {\cal B}(t) \|^2_{L^{\infty}}R^{-\varepsilon} ) \leq C (\|G(t)\|^2_{H^1} + R^{-\varepsilon}).
\end{array}
$$

\noindent Substitution into (\ref{E4.16}) and using Gronwall's inequality yields

\be \diint v^a_0\, \tilde{g}^2 \, dvdx \leq C \dint^t_0 \|G(\tau)\|^2_{H^1} d\tau + CR^{-\varepsilon} t. \label{E4.17}
\ee

Again proceeding as in (\ref{E4.10}) and (\ref{E4.11}) we have

\be
\begin{array}{rl} & \dfrac{d}{dt} \diint v^b_0 (\partial_x \tilde{g})^2 \, dvdx = -2 \diint v^b_0 | \nabla_v \partial_x \tilde{g}|\, dvdx\\
\\
& + \diint (\partial_x \tilde{g})^2 ( \Delta_v v^b_0 + K \cdot \nabla_v v^b_0) \, dvdx \\
\\
& + 2 \diint (\partial_x \tilde{F} (K - {\cal K}) + \tilde{g}\, \partial_x K + \tilde{F} \partial_x (K - {\cal K})) \cdot (v^b_0 \nabla_v \partial_x \, \tilde{g} + \partial_x \tilde{g}\nabla_v v^b_0) \, dvdx\\
\\
\leq & C \diint v^b_0 ( \partial_x \tilde{g})^2\, dvdx +  C \diint (\partial_x \tilde{F} )^2 |K-{\cal K}|^2 v^b_0\, dvdx\\
\\
& + C \diint \tilde{g}^2 |\partial_x K|^2 v^b_0 \, dvdx + C \diint \tilde{F}^2 | \partial_x (K-{\cal K}) | ^2 v^b_0 \, dvdx.
\end{array}
\label{E4.18}
\ee

\noindent By (\ref{E4.14}), (\ref{E4.13}) and \change{the} Sobolev \change{embedding theorem} we have

\be
\begin{array}{rl}
& \diint (\partial_x \tilde{F})^2 |K- {\cal K}|^2 v^b_0 \, dvdx \\
\\
\leq & C \diint (\partial_x \tilde{F})^2 (|G|^2 + |{\cal B}|^2 R^{-\varepsilon} ) v^{\beta}_0\, dvdx  \\
\\
\leq & C ( \|G(t)\|^2_{H^1} + \|{\cal B}(t)\|^2_{H^1} R^{-\varepsilon}) \leq C(\|G(t)\|^2_{H^1} + R^{-\varepsilon}).
\end{array}
\label{E4.19}
\ee

\noindent Note that (using (\ref{E4.17}))

\be
\begin{array}{rcl}
\dint \tilde{g}^2 v^{b+2}_0 dv & \leq & 2 \diint |\tilde{g}||\partial_x \tilde{g}| v^{\frac{a+b}{2}}_0 \, dvdx \\
\\
& \leq & 2 \left(\diint \tilde{g}^2 v^a_0 \, dvdx \right)^{\frac{1}{2}} \left(\diint (\partial_x \tilde{g})^2 v^b_0 \, dvdx \right)^{\frac{1}{2}}\\
\\
& \leq & C \dint^t_0 \|G(\tau)\|^2_{H^1} d \tau + CR^{-\varepsilon} t + \diint (\partial_x \tilde{g})^2 v^b_0 \, dvdx
\end{array}
\label{E4.20}
\ee

\noindent so by (\ref{E4.20}), (\ref{E4.7}), and \change{the} Sobolev \change{embedding theorem}

\be
\begin{array}{rl}
& \diint \tilde{g}^2 |\partial_x K|^2 v^b_0\, dvdx \leq \diint \tilde{g}^2 (|\partial_x E|^2 + (\partial_x B)^2)v^{b+2}_0\, dvdx \\
\\
\leq & C \dint^t_0 \|G(\tau)\|^2_{H^1} d\tau + CR^{-\varepsilon} t + C \diint (\partial_x \tilde{g})^2 v^b_0 \, dvdx.
\end{array}
\label{E4.21}
\ee

\noindent Using (\ref{E4.15}), (\ref{E4.12}), (\ref{E4.13}), and (\ref{E4.7}) we have

\be
\begin{array}{rl}
& \diint \tilde{F}^2 |\partial_x (K - {\cal K})|^2 v^b_0\, dvdx\\
\\
\leq & \diint \tilde{F}^2 (|\partial_x G|^2 + (\partial_x {\cal B})^2 R^{-\varepsilon}) v^{\beta}_0\, dvdx\\
\\
\leq & C \dint (|\partial_x G|^2 + (\partial_x {\cal B})^2 R^{-\varepsilon}) \, dx\\
\\
\leq & C \|G(t)\|^2_{H^1} + CR^{-\varepsilon}.
\end{array}
\label{E4.22}
\ee

\noindent Substitution of (\ref{E4.19}), (\ref{E4.21}), and (\ref{E4.22}) into (\ref{E4.18}) yields

$$
\begin{array}{l}
\dfrac{d}{dt} \diint v^b_0 (\partial_x \tilde{g})^2 \, dvdx \leq C \diint v^b_0 (\partial_x \tilde{g})^2\, dvdx\\
\\
\hspace*{.5in} + C \|G(t)\|^2_{H^1} + CR^{-\varepsilon} + C \dint^t_0 \|G(\tau)\|^2_{H^1}\, d\tau.
\end{array}
$$

\noindent By Gronwall's inequality we have

\be
\diint v^b_0 (\partial_x \tilde{g})^2 \, dvdx \leq C \dint^t_0 \|G(\tau)\|^2_{H^1}\, d\tau + CR^{-\varepsilon} t.
\label{E4.23}
\ee

Next we consider the fields.  We have

$$
\dfrac{d}{dt} \dint (|\tilde{E}|^2 + \tilde{B}^2 ) dx = - 2 \dint \tilde{E} \cdot \tilde{j} dx.
$$

\noindent By (\ref{E4.9})

$$
\dint |\tilde{j}|^2 dx \leq \dint \left( \dint \tilde{f} v^{\alpha}_0\, dv \right) \left(\dint v^{-\alpha}_0 |v|^2\, dv \right) dx \leq C
$$

\noindent so

$$
\dfrac{d}{dt} \dint \left( |\tilde{E}|^2 + \tilde{B}^2 \right)\, dx \leq C \left(\dint|\tilde{E}|^2dx \right)^{\frac{1}{2}}
$$

\noindent and

\be \dint \left(|\tilde{E}|^2 + \tilde{B}^2 \right) dx \leq C_0 + Ct \label{E4.24}
\ee

\noindent follows.  Similarly by (\ref{E4.13})

$$\dint |\partial_x \tilde{j}|^2 \, dx \leq \dint \left( \dint (\partial_x \tilde{f} )^2 v^{\beta}_0 dv \right) \left( \dint v^{2-\beta}_0 dv \right) \, dx \leq C
$$

\noindent so

$$
\begin{array}{rcl}
\dfrac{d}{dt} \dint \left( |\partial_x \tilde{E}|^2 + \left(\partial_x \tilde{B}\right)^2 \right)\, dx
& \leq & 2 \| \partial_x \tilde{E} (t)\|_{L^2} \| \partial_x \tilde{j}(t) \|_{L^2}\\
\\
& \leq & C \| \partial_x \tilde{E} (t) \|_{L^2}
\end{array}
$$

\noindent and

\be \dint \left( |\partial_x \tilde{E}|^2 + (\partial_x \tilde{B})^2 \right) dx \leq C_0 + Ct
\label{E4.25}
\ee

\noindent follows.  \change{In the same manner} (\ref{E4.17}) yields

$$
\begin{array}{rcl}
\dint |\tilde{j} - \tilde{J}|^2 dx & \leq & \dint \left( \dint \tilde{g}^2 v^a_0\, dv \right) \left( \dint v^{-a}_0 |v|^2\, dv \right)\, dx \\
\\
& \leq & C \dint^t_0 \|G(\tau)\|^2_{H^1} d\tau + CR^{-\varepsilon} t
\end{array}
$$

\noindent and, letting $\tilde{G} = (\tilde{E}, \tilde{B}) - (\tilde{\cal E}, \tilde{\cal B})$,

$$
\dint |\tilde{G}|^2 dx \leq C \dint^t_0 \|G(\tau)\|^2_{H^1} + CR^{-\varepsilon}t
$$

\noindent follows.  Lastly (\ref{E4.23}) yields

$$
\begin{array}{rcl}
\dint | \partial_x (\tilde{j}-\tilde{J}) |^2 \, dx & \leq & C \dint \left( \dint (\partial_x \tilde{g})^2 v^b_0\, dv \right) \left( \dint v^{2-b}_0 dv\right) \, dx \\
\\
& \leq & C \dint^t_0 \|G(\tau)\|^2_{H^1}\, d \tau + CR^{-\varepsilon}t
\end{array}
$$

\noindent and

$$ \dint |\partial_x \tilde{G} |^2 \, dx \leq C \dint^t_0 \|G(\tau)\|^2_{H^1} d\tau + CR^{-\varepsilon}t
$$

\noindent follows.

From (\ref{E4.24}) and (\ref{E4.25}) we have

$$
\| (\tilde{E}, \tilde{B}) (t) \|^2_{H^1} \leq 2 C_0 + CT \leq 10C_0
$$

\noindent for $T$ suitably restricted.  Also,

\be
\begin{array}{rl}
& \| (\tilde{f} - \tilde{F}) (t) \|^2_{\cal F} + \| (\tilde{E},\tilde{B})(t) - (\tilde{\cal E}, \tilde{\cal B}) (t)\|^2_{H^1}\\
\\
\leq & C \dint^t_0 \|(E,B)(\tau) - ({\cal E}, {\cal B})(\tau)\|^2_{H^1} d\tau + CR^{-\varepsilon} t.
\end{array}
\label{E4.26}
\ee

Define $(f^{n+1}, E^{n+1}, B^{n+1}) = {\cal L}^{2^n} (E^n,B^n)$ for $n \geq 0$ where $f^0,E^0,B^0$ are determined by the initial conditions.  By (\ref{E4.26}) we have

\be \begin{array}{rl}
& \| (f^{n+1} - f^n ) (t)\|^2_{\cal F} + \| (E^{n+1},B^{n+1}) (t) - (E^n , B^n) (t) \|^2_{H^1}\\
\\
\leq & C \dint^t_0 \|(E^n,B^n) (\tau) - (E^{n-1}, B^{n-1}) (\tau)\|^2_{H^1}\, d\tau + C 2^{-n\varepsilon}t.
\end{array}
\label{E4.27}
\ee

Suppose $A > 1$ and

$$
0 \leq x^{n+1} (t) \leq C \left( \dint^t_0 x^n(\tau) d\tau + A^{-n}t \right)
$$

\noindent for $n \geq 0$.  Then by induction

$$
\begin{array}{rcl}
x^n(t) & \leq & \left( \|x^0\|_{L^{\infty}}+1 \right) \dfrac{(Ct)^n}{n!} + A^{-n} \overset{n-1}{\underset{\ell =1}{\sum}} \dfrac{(CAt)^{\ell}}{\ell !}\\
\\
& \leq & \left(\| x^0\|_{L^{\infty}} + 1 \right) \dfrac{(CT)^n}{n!} + A^{-n} e^{CAT}.
\end{array}
$$

\noindent Since this bound is summable, it follows from (\ref{E4.27}) that $(E^n,B^n)$ is Cauchy in $C([0,T]; H^1)$ and $f^n$ is Cauchy in $C([0,T]; {\cal F})$.  Let $(f, E,B) = \change{\lim_{n \to \infty}} (f^n, E^n, B^n)$.

We will now use the explicitly known fundamental solution for the linear equation

$$
\partial_t f + v_1 \partial_x f = \Delta_v f,
$$

\noindent namely, for $0 \leq \tau < t,\ x,y \in \mathbb{R},v, w \in \mathbb{R}^2$

$$\begin{array}{rl}
& {\cal G}(t,x,v,\tau ,y,w)\\
 \\
=  & [ 4\pi(t-\tau)]^{-1}e^{\frac{-|v-w|^2}{4(t-\tau)}} \left[ \dfrac{\pi}{3} (t-\tau)^3\right]^{-\frac{1}{2}} \exp \left(-3 \dfrac{(x-y - \frac{1}{2} (t-\tau)(v_1+w_1))^2}{(t-\tau)^3}\right).
\end{array}
$$

\noindent This may be derived from line (2.5) of \cite{16} by letting $\beta \ra 0^+$ (with $N=2$) and then integrating in $y_2$.  It may also be derived directly by Fourier transform.  Since

$$
\left\{ \begin{array}{l} \partial_t f^{n+1} + v_1 \partial_x f^{n+1} = \Delta_v f^{n+1} - K^n \cdot \nabla_v f^{n+1}\\
\\
f^{n+1} (0, \cdot, \cdot) = f^0
\end{array} \right.
$$

\noindent it follows by Theorems II.2 and II.3 of \cite{16} that

$$
\left. f^{n+1} = H + \dint^t_0 \diint {\cal G}(t,x,v,\tau ,y,w)(-K^n \cdot \nabla_w f^{n+1})\right|_{(\tau , y,w)} \, dw dy d\tau
$$

\noindent where

$$
H(t,x,v) = \diint {\cal G} (t,x,v,0,y,w) f^0 (y,w) \, dw dy.
$$

\noindent It is easy to check that

$$\diint \left| \nabla_w {\cal G}(t,x,v,\tau , y, w) \right|\, dwdy \leq C(t-\tau)^{-\frac{1}{2}}
$$

\noindent and it follows that

$$
\left. f^{n+1} = H + \dint^t_0 \diint \nabla_w {\cal G} (t,x,v,\tau , y, w) \cdot (K^n f^{n+1}) \right|_{(\tau , y,w)} \, dw dy d\tau.
$$

%
%
%
%
%
%

\noindent By Lemma 2.1

$$
0 \leq f^n \leq \sup f^0
$$

\noindent so

\be 0 \leq f \leq \sup f^0
\label{E4.29}
\ee

\noindent follows.

Thus far we have assumed $(f^0, E^0_2, B^0)$ to be smooth.  Now consider $ (f^0, E^0_2, B^0)$ as in Theorem 1.1.  Consider a sequence $(f^{0k}, E^{0k}_2, B^{0k})$ of smooth initial conditions with

$$
\| v^{\frac{\alpha}{2}}_0 (f^{0k} - f^0 ) \|_{L^2} + \| v^{\frac{\beta}{2}}_0 \partial_x (f^{0k} - f^0) \|_{L^2} + \| (E^{0k}_0,B^{0k} ) - (E^0_2, B^0) \|_{H^1} \ra 0.
$$

\noindent By a limiting procedure like the above we conclude that (\ref{E4.29}) holds for $(f^0, E^0_2, B^0)$ as in Theorem 1.1.

By Theorem II.3 of \cite{16} $H \in C([0, \infty) \times \mathbb{R}^3) \cap C^1 ((0,\infty) \times \mathbb{R}^3)$ with $\partial_{v_i} \partial_{v_j} H$ continuous on $t > 0$ and

$$
\left\{ \begin{array}{rcl} \partial_t H + v_1 \partial_x H & = & \Delta_v H\\
\\
 H(0,\cdot, \cdot) & = & f^0.\end{array} \right. 
 $$

Next, we will show that $f$ is differentiable in $v$ and H\"{o}lder continuous in $x$.
Proceeding  as in (4.16)  we have
\begin{equation}
\begin{array}{lll}
\frac{d}{dt} \iint v^a_0 \left(f^{n+1}-f^{\ell+1}\right)^2 dvdx = -2 \iint v^{a}_0
\left | \nabla_v\left(f^{n+1}-f^{l+1}\right) \right|^2 dvdx \\
\hspace*{2in} + \iint \left(f^{n+1}-f^{\ell+1}\right)^2 \left(\Delta_v v^a_0 + E^n \cdot \nabla_v v^a_0\right) dvdx  \\
\hspace*{2in} + 2 \iint f^{\ell+1 }\left(K^n-K^{\ell} \right)  \cdot \left(v^a_0 \nabla_v\left(f^{n+1}-f^{\ell+1}\right) \right. \\
\hspace*{2in} +  \left.  \left(f^{n+1}-f^{\ell+1}\right) \nabla_vv^a_0\right) dvdx.  \label{Eq2}
\end{array}
\end{equation}
We will use the notation
$$\sigma^{n,l} \leq o^{n,l}$$
to mean $\forall \epsilon>0 \exists N $ such that $ n, \ell > N \Rightarrow \sigma^{n,l} \leq \epsilon$.  Recall that
$$
\begin{array}{lll}
\underset{t}{\sup} \left|\left| \left(E^nB^n\right) - \left(E^{\ell},B^{\ell} \right)
 \right| \right|_{H^1(\mathbb{R})} + \\
\nonumber \\
\hspace*{.25in} +  \underset{t}{\sup} \iint v^a_0 \left(f^{n+1}-f^{\ell+1}\right)^2 dvdx \leq o^{n,\ell}. \nonumber
\end{array}
$$
By (4.14) we have
$$\begin{array}{lll}
\left|K^n-K^{\ell}\right|&  \leq &  v_0 \left|\left(E^n,B^n\right)-\left(E^{\ell},B^{\ell}\right) \right| + v^{1 + \epsilon/2}_0 \left|B^{\ell}\right| 2^{-n \epsilon/2}\\
& \leq & v^{1 + \epsilon/2}_0 o^{n,\ell}\end{array}
$$
where we assume $\ell > n$.  

Recall that $\alpha = a+2 + \epsilon$ and note that
$$\begin{array}{lll}
\iint\left|f^{\ell + 1} \left(K^n-K^{\ell}\right) \cdot v^a_0 \nabla_v\left(f^{n+1}-f^{\ell +1}\right)\right| dvdx\\
 \leq o^{n,\ell} \iint \left|f^{\ell+1}\right|v_0^{a+1+ \epsilon/2} \left| \nabla_v \left(f^{n+1}-f^{\ell+1}\right)\right|  dvdx  \\
 \\
 \leq o^{n,\ell} \sqrt{\iint v^{\alpha}_o\left(f^{\ell+1}\right)^2dvdx} \ \sqrt{\iint v^a_0\left| \nabla_v\left(f^{n+1}-f^{\ell + 1}\right) \right|^2 dvdx}\\
\\
\leq o^{n,\ell} \sqrt{\iint v^a_0 \left| \nabla_v\left(f^{n+1}-f^{\ell +1}\right)\right|^2 dvdx}.
\end{array}
$$
Similarly,
$$\begin{array}{lll}
\iint\left|f^{\ell+1} \left(K^n-K^{\ell}\right) \cdot \left(f^{n+1}-f^{\ell+1}\right) \nabla_vv^a_0\right|dvdx\\
\\
\leq o^{n,\ell} \sqrt{\iint v^a_0 \left(f^{n+1}-f^{\ell+1}\right)^2dvdx} \leq o^{n,\ell}. \end{array}
$$
So, by \eqref{Eq2} we find
$$\begin{array}{lll}
\frac{d}{dt} \iint v^a_0 \left(f^{n+1}-f^{\ell+1}\right)^2 dvdx \leq -\iint v^a_0\left|\nabla_v\left(f^{n+1}-f^{\ell+1}\right)\right|^2 dvdx\\
\\
+ o^{n,\ell} +   \left[o^{n,\ell}\sqrt{\iint v^a_0\left|\nabla_v\left(f^{n+1}-f^{\ell+1}\right)\right|^2}dvdx\right.\\
 \\
 -  \left.  \iint v^a_0 \left|\nabla_v\left(f^{n+1}-f^{\ell+1}\right) \right|^2dvdx \right]\\
\\
\leq o^{n,\ell} - \iint v^a_0\left|\nabla_v\left(f^{n+1}-f^{\ell+1}\right)\right|^2dvdx. \end{array}
$$
It follows that 
\begin{equation} \label{Eq3}
\int^T_0 \iint v^a_0 \left| \nabla_v \left(f^{n+1}-f^{\ell+1}\right)\right|^2 dvdxd\tau \leq o^{n, \ell}. 
\end{equation} 
In a similar manner we may show that 
\begin{equation} \label{Eq4} 
\int^T_0 \iint v^{\alpha}_0 \left|\nabla_v f^{n+1}\right|^2 dvdxd\tau \leq C. 
\end{equation} 
Note that the exponent of $v_0$ in \eqref{Eq3} is $a$, but in \eqref{Eq4} it is $\alpha$.  

Next, we derive $L^p$ bounds on ${\cal G}$.  Considering $0 \leq \tau \leq t \leq T$ and letting $p \geq 1,\ b, \theta \geq 0$ and 
$$u = \frac{v-w}{\sqrt{t-\tau}} \ \ z = \frac{x-y- \frac{t-\tau}{2}(v_1+w_1)}{(t-\tau)^{3/2}}$$ 
we have 
$$\begin{array}{llll}
 \iint w^{-b \theta}_0 {\cal G}^pdwdy\\
=  C \iint \left(\sqrt{1+|v - \sqrt{t- \tau} u|^2}\right)^{-b \theta} \left[(t-\tau)^{-5/2}e^{-|u|^2/4}e^{-3z^2}\right]^p (t-\tau)^{5/2} dzdu\\
\\
= C(t-\tau)^{\frac{5}{2}(1-p)} \left({\int}_{|u| < \frac{1}{2}(t- \tau)^{-1/2}|v|}
{\left(\sqrt{1+|v- \sqrt{t-\tau}u|^2}\right)}^{-b \theta}e^{-p|u|^2/4} du  \right. \\
\\
 \left. + \int_{|u| > \frac{1}{2}(t- \tau)^{-1/2}|v|} \left(\sqrt{1+|v-\sqrt{t-\tau}u|^2}\right)^{-b \theta}e^{-p|u|^2/4} du\right)\\
\\
\leq C (t-\tau)^{\frac{5}{2}(1-p)} \left[\left(\sqrt{1+\left(\frac{1}{2}|v|\right)^2}\right)^{-b \theta} \int e^{-|u|^2/4} du + e^{-\frac{1}{8}\left(\frac{|v|}{2 \sqrt{t-\tau}}\right)^2} \int e^{-|u|^2/8} du \right] \\
\\
\leq C(t- \tau)^{\frac{5}{2}(1-p)}\left[\left(\sqrt{1+|v|^2}\right)^{-b \theta} + e^{-C|v|^2}\right] \\
\\
\leq C(t-\tau)^{\frac{5}{2}(1-p)}v^{-b \theta}_0. \end{array} 
$$
So for $1 \leq p < 7/5$ we have 
\begin{equation} \label{Eq5} 
\int^t_0 \iint w^{-b \theta}_0 {\cal G}^p dwdy d \tau \leq C v^{-b \theta}_0. 
\end{equation} 
Here, constants may depend on $p,b$, and $\theta$.  Later, specific choices of $p,b$, and $\theta$ are used and this dependence is removed. 

Next, we bound $|\nabla_w {\cal G}|$ and $|\nabla_v{\cal G}|$.  Note that 
$$\begin{array}{lll} 
\iint w^{-b \theta}_0 \left(\frac{|v-w|}{t-\tau}{\cal G}\right)^p dwdy\\
= C \iint \left(\sqrt{1+|v- \sqrt{t-\tau}u|^2}\right)^{-b\theta}
\left[\frac{|u|}{\sqrt{t-\tau}}(t-\tau)^{-5/2}e^{-|u|^2/4}e^{-3z^2}\right]^p(t-\tau)^{5/2}dzdu \\
= C(t-\tau)^{\frac{5}{2}-3p}\int \left(\sqrt{1+|v-\sqrt{t-\tau}u|^2}\right)^{-b \theta}\left[|u|e^{-|u|^2/4}\right]^pdu\\
\\
\leq C(t-\tau)^{\frac{5}{2} -3p}v^{-b \theta}_0 \end{array} 
$$
and similarly, 
$$
\iint w^{-b \theta}_0 \left( \frac{\left|x-y-\frac{t-\tau}{2} (v_1+w_1)\right|}{(t-\tau)^2}{\cal G}\right)^p dwdy
\leq C(t-\tau)^{\frac{5}{2}-3p}v_0^{-b \theta}.
$$
Hence,
$$\iint w_0^{-b \theta}\left(|\nabla_w {\cal G}|^p+|\nabla_v{\cal G}|^p \right) dwdy \leq C (t-\tau)^{\frac{5}{2}-3p}v_0^{-b \theta} $$
and for $1 \leq p < 7/6$ 
\begin{equation} \label{Eq6} 
\int^t_0 \iint w_0^{-b \theta} \left(|\nabla_w {\cal G}|^p + |\nabla_v{\cal G}|^p\right)dwdyd \tau\leq C v_0^{-b \theta}.
\end{equation}

In a very similar manner it may be shown that 
\begin{equation} \label{Eq7} 
 \int^T_0 \iint v_0^{-b\theta}{\cal G}^p dvdxdt \leq C w_0^{-b \theta} 
\end{equation}
for $p < 7/5$ and
$$ \int^T_0 \iint v_0^{-b \theta}\left(|\nabla_v {\cal G}|^p+|\nabla_w {\cal G}|^p\right) dvdxdt \leq Cw_0^{-b \theta}$$
for $p<7/6.$

Next we derive two inequalities that will be used repeatedly.  Let $p, q \in [1, \infty)$ and $r \in [1, \infty]$ satisfy 
$$\frac{1}{p} + \frac{1}{q} -1 = \frac{1}{r}. $$ 
We will first consider the case $ r \neq \infty$, but what follows may be easily adapted to the case $r = \infty$.  Let $\theta \geq 0$ and define 
$$ b = \left( \frac{1}{p} - \frac{1}{r}\right)^{-1} \hspace*{.25in} c = \left( \frac{1}{q} - \frac{1}{r} \right)^{-1} $$
and note that 
$$ \frac{1}{r} + \frac{1}{b} + \frac{1}{c} = 1. $$ 
By H\"older's inequality and using \eqref{Eq5}, we have for $p < 7/5$ and any $h (\tau,y,w) \geq 0$
$$\begin{array}{llll}
&& \int^t_0 \iint {\cal G} hdwdyd \tau\\
\\
& = & \int^t_0\iint\left[{\cal G}^{p/r}(w^{\theta}_0h)^{q/r}\right] \left[{\cal G}^{p\left(\frac{1}{p}-\frac{1}{r}\right)} w_0^{-\theta}\right] \left[w_0^{\theta}h\right]^{q \left(\frac{1}{q}-\frac{1}{r}\right)}dwdyd \tau\\
& \leq & \left(\int^t_0 \iint {\cal G}^p\left(w_0^{\theta}h\right)^qdwdyd \tau\right)^{1/r} 
 \left(\int^t_0 \iint {\cal G}^pw_0^{-b \theta}dwdyd \tau \right)^{1/b} \\
&&\hspace*{0.1in}  \left( \int^t_0 \iint \left(w_0^{\theta}h\right)^qdwdyd \tau \right)^{1/c} \\
& \leq & Cv^{-\theta}_0 \left( \int^t_0 \iint {\cal G}^p\left(w_0^{\theta}h\right)^q dwdy d \tau \right)^{1/r} \left( \int^t_0\iint \left(w_0^{\theta}h\right)^q dwdy d \tau \right)^{1/c}. \end{array} 
$$
Hence, using \eqref{Eq7} we have, for $p < 7/5$
\begin{equation} \begin{array}{lll} \label{Eq8} 
\left[\int^T_0 \iint \left(v^{\theta}_0 \int^t_0 \iint {\cal G}hdwdyd\tau\right)^r dvdxdt\right]^{1/r} \\
\leq C \left[ \int^T_0 \iint \left(\int^t_0 \iint {\cal G}^p \left(w^{\theta}_0h\right)^q dwdyd \tau\right) \right.\\
\hspace*{.1in}\left.  \left( \int^t_0 \iint \left(w^{\theta}_0h\right)^q dwdy d \tau \right)^{r/c} dvdxdt  \right]^{1/r} \\
\leq C \left[\left(\int^T_0 \iint \left(w^{\theta}_0h\right)^q dwdyd \tau \right)^{r/c} \right. \\
\hspace*{.1in}\left.  \int^T_0 \iint \left(\int^T_0 \iint {\cal G}^pdvdxdt\right)\left(w^{\theta}_0h\right)^q dwdyd \tau\right]^{1/r}\\
\leq C \left[\left(\int^T_0 \iint \left(w^{\theta}_0h\right)^q dwdy d \tau \right)^{1+ r/c}\right]^{1/r} \\
= C \left[ \int^T_0 \iint \left(w^{\theta}_0h \right)^q dwdyd \tau \right]^{1/q}. \end{array} 
\end{equation}
Similarly, using \eqref{Eq6} we find for $p < 7/6$
\begin{equation} \begin{array}{lll} \label{Eq9} 
\left[ \int^T_0 \iint \left(v^{\theta}_0 \int^t_0 \iint \left[|\nabla_w{\cal G}|+|\nabla_v{\cal G}|\right] hdwdyd \tau \right)^r dvdxdt \right]^{1/r} \\
\leq C \left( \int^T_0 \iint \left(w^{\theta}_0h\right)^q dwdy d \tau \right)^{1/q}. \end{array} 
\end{equation} 

Now we will show that $f^n$ converges in $L^{\infty}$.  We have 
\begin{equation} \label{Eq10} 
f^{n+1}-f^{\ell+1}= - \int^t_0 \iint {\cal G}\left(K^n \cdot \nabla_vf^{n+1}-K^{\ell} \cdot \nabla_vf^{\ell +1}\right) (\tau,y,w) dwdyd \tau. 
\end{equation} 
Using (4.14) and taking $\ell > n$
\begin{equation} \begin{array}{llll} \label{Eq11} 
\left|K^n \cdot \nabla_v f^{n+1}- K^{\ell} \cdot \nabla_v f^{\ell+1}\right| \leq \left|K^n-K^{\ell}\right| \left|\nabla_vf^{n+1}\right|\\
 \hspace*{.25in}  + \left|K^{\ell}\right| \left|\nabla_v \left(f^{n+1}-f^{\ell +1}\right)\right|\\
 \leq \left|\left(E^n,B^n\right)-\left(E^{\ell},B^{\ell}\right)\right| w_0 \left|\nabla_vf^{n+1}\right| \\
 + Cw_0^{1+ \epsilon/2}2^{-n \epsilon/2} \left|\nabla_vf^{n+1}\right| + C w_0 \left|\nabla_v
 \left(f^{n+1}-f^{\ell + 1}\right) \right|\\
 \leq o^{n, \ell}w_0^{1+ \epsilon/2}\left|\nabla_vf^{n+1}\right| + Cw_0\left|\nabla_v\left(f^{n+1}-f^{\ell+1}\right) \right|. \end{array} 
 \end{equation} 
Next, we will apply \eqref{Eq8} with $p < 7/5$ and $q=2$.
Note that by taking $p$ close to $7/5$, we may make $r = \left( \frac{1}{p} + \frac{1}{q} -1 \right)^{-1}$ close to 
$$\left[\frac{5}{7} + \frac{1}{2} -1 \right]^{-1} = 14/3. $$
Thus, applying \eqref{Eq8} with $p < 7/5,\ q = 2, \ {\theta} = \frac{a}{2}-1$, and 
$$h = w_0\left| \nabla_v\left(f^{n+1}-f^{\ell+1}\right) \right|$$
while using \eqref{Eq3} yields
$$\begin{array}{lll}
\left[\int^T_0 \iint \left(v^{\frac{a}{2}-1}_0 \int^t_0 \iint {\cal G} w_0 \left| \nabla_v\left(f^{n+1}- f^{\ell+1} \right) \right| dwdy d \tau \right)^r dvdxdt \right]^{1/r}\\
 \leq C \left[\int^T_0 \iint \left( w_0^{\frac{a}{2}-1}w_0 \left|\nabla_v \left(f^{n+1}-f^{\ell+1}\right)\right| \right)^2 dwdyd \tau\right]^{\frac{1}{2}} \leq o^{n, \ell} \end{array}  
$$ 
for $r < 14/3$.
Applying \eqref{Eq8} with $p<7/5$, $q = 2$, $\theta = \frac{\alpha}{2}-1-\epsilon/2$, and $h=w_0^{1+ \epsilon/2} \left| \nabla_vf^{n+1}\right|$, and then using \eqref{Eq4} yields
$$ \begin{array}{llll} 
\left[\int^T_0 \iint \left(v_0^{\frac{\alpha}{2} -1-\epsilon/2}\int^t_0 \iint {\cal G}w_0^{1+ \epsilon/2} \left|\nabla_vf^{n+1}\right| dwdy d \tau \right)^ rdvdxdt\right]^{1/r} \\
 \leq C \left[\int^T_0 \iint \left(w_0^{\frac{\alpha}{2} -1 - \epsilon/2}w_0^{1+ \epsilon/2}\left| \nabla_vf^{n+1}\right| \right)^2 dwdyd \tau\right]^{\frac{1}{2}} \leq C \end{array}
$$
for $r < 14/3$.
Since $\frac{\alpha}{2}-1-\frac{\epsilon}{2} > \frac{a}{2}-1$, \eqref{Eq10} and \eqref{Eq11} now yield for $r < 14/3$
\begin{equation} \label{Eq12}
\int^T_0 \iint\left(v^{\frac{a}{2}-1}_0 \left|f^{n+1}-f^{\ell+1}\right|\right)^r dvdxdt \leq o^{n.\ell}. 
\end{equation}

\noindent Similarly, using \eqref{Eq4} we may show that  for $r < 14/3$
\begin{equation} \label{Eq13} 
\int^T_0 \iint \left (v_0^{\frac{\alpha}{2}-1}\left|f^{n+1}\right|\right)^r dvdxdt \leq C. 
\end{equation} 

To use \eqref{Eq12} we integrate by parts in \eqref{Eq10} to obtain 
\begin{equation} \label{Eq14} 
f^{n+1}-f^{\ell +1} = \int^t_0 \iint \nabla_w {\cal G}\cdot\left(K^nf^{n+1}-K^{\ell}f^{\ell+1}\right) dwdy d \tau 
\end{equation} 
and using (4.14)
\begin{equation} \begin{array}{llll} \label{Eq15} 
\left|K^nf^{n+1}-K^{\ell}f^{\ell+1}\right| \leq \left|\left(E^n,B^n\right) - \left(E^{\ell}, B^{\ell}\right) \right| w_0 \left|f^{n+1}\right|\\
+C w_0^{1+ \epsilon/2} 2^{-n\epsilon/2} \left|f^{n+1}\right|+Cw_0\left|f^{n+1}-f^{\ell+1}\right|\\
\leq o^{n, \ell}w_0^{1+ \epsilon/2}\left|f^{n+1}\right| + Cw_0 \left|f^{n+1}-f^{\ell+1}|\right. .\end{array} 
\end{equation}

Next, we will apply \eqref{Eq9}, but now with $p < 7/6$ and $q < 14/3$.
Note that we may take $p$ close to $7/6$ and $q$ close to $14/3$ to make $r$ close to 
$$\left(\frac{6}{7} + \frac{3}{14} -1 \right)^{-1} = 14.$$
Thus, applying \eqref{Eq9} with $p < 7/6,\ q < 14/3,\ \theta = \frac{a}{2}-2$, and 
$$h=w_0\left|f^{n+1}-f^{\ell+1}\right| $$
and then using \eqref{Eq12} yields
$$\begin{array}{lllll} 
\left[\int^T_0 \iint \left(v_0^{\frac{a}{2}-2}\int^t_0 \iint \left|\nabla_w {\cal G}\right| w_0 \left|f^{n+1}-f^{\ell+1}\right|dwdyd \tau \right)^r dvdxdt\right]^{1/r}\\
\leq C\left[\int^T_0 \iint \left(w_0^{\frac{a}{2}-2}w_0\left|f^{n+1}-f^{\ell+1}\right|\right)^q dwdyd \tau\right]^{1/q} \leq o^{n\ell} \end{array} 
$$
for $r < 14$.
Applying \eqref{Eq9} again with $p < 7/6$, $q < 14/3$, $\theta = \frac{\alpha}{2} - 2 - \frac{\epsilon}{2}$ and $h = w^{1+ \epsilon/2}_0\left|f^{n+1}\right|$ yields 
$$\begin{array}{lll}
\left[\int^T_0 \iint \left(v_0^{\frac{\alpha}{2}-2-\epsilon/2} \int^t_0 \iint \left|\nabla_w {\cal G}\right|w_0^{1+ \epsilon/2}\left|f^{n+1}\right| dwdy d \tau \right)^r dvdxdt\right]^{1/r}\\
\leq\left[\int^T_0 \iint \left(w_0^{\frac{\alpha}{2}-1}\left|f^{n+1}\right|\right)^q dwdyd \tau \right]^{1/q} \leq C \end{array} 
$$
for $r < 14$ by using \eqref{Eq13}.
With these two estimates, \eqref{Eq14} and \eqref{Eq15} yield 
$$\int^T_0 \iint \left(v_0^{\frac{a}{2}-2}\left|f^{n+1}-f^{\ell+1}\right|\right)^r dvdxdt \leq o^{n,\ell}$$
for $r < 14$. Similarly,
$$\int^T_0 \iint \left(v_0^{\frac{\alpha}{2}-2}\left|f^{n+1}\right|\right)^r dvdxdt \leq C.$$

Finally, we can apply \eqref{Eq9} with $p = \frac{7}{6.4}, q=\frac{14}{1.2}$, and $r= \infty$.  Proceeding as above we obtain 
$$\left|\left| v_0^{\frac{a}{2}-3}\left (f^{n+1}-f^{\ell+1}\right) \right|\right|_{L^{\infty}} \leq o^{n,\ell}.$$
Recalling 
$$f= \underset{n \rightarrow \infty}{\lim} f^n$$ 
it follows that 
$$\left|\left| v_0^{\frac{a}{2}-3}\left(f^n-f\right) \right|\right|_{L^{\infty}} \rightarrow 0 \ \ {\rm as} \ \ n \rightarrow \infty. $$

Next, we bound $\nabla_vf^n$ in $L^{\infty}$.  We have 
$$f^{n+1} = H - \int^t_0 \iint {\cal G} K^n \cdot \nabla_vf^{n+1} dwdy d \tau$$ 
so 
\begin{equation} \label{Eq16}
\begin{array}{lllll}
\left|\nabla_v\left(f^{n+1}-H\right)\right| = \left|\int^t_0 \iint \nabla_v {\cal G}K^n \cdot \nabla_vf^{n+1} dwdy d \tau\right|\\
\leq C \int^t_0 \iint \left| \nabla_w {\cal G}\right| w_0 \left|\nabla_vf^{n+1}\right| dwdy d \tau. \end{array} 
\end{equation} 
Applying \eqref{Eq9} with $p < 7/6,\ q = 2,\ \theta = \frac{\alpha}{2}-1$, and $h = w_0\left|\nabla_vf^{n+1}\right|$, and then using \eqref{Eq4} yields
$$\begin{array}{lllll} 
\left[\int^T_0 \iint \left( v_0^{\frac{\alpha}{2}-1}\int^t_0 \iint \left|\nabla_w{\cal G}\right| w_0 \left|\nabla_vf^{n+1}\right|dwdyd \tau \right)^r dvdxdt \right]^{1/r}\\
\leq \left[\int^T_0 \iint \left(w_0^{\frac{\alpha}{2}-1}w_0 \left|\nabla_vf^{n+1}\right|\right)^2 dwdy d \tau \right]^{1/2} \leq C. \end{array} 
$$
Note that taking $p$ close to $7/6$ yields $r$ close to 
$$\left(\frac{6}{7} + \frac{1}{2} -1 \right)^{-1} = 14/5. $$ 
Hence, using \eqref{Eq16} we find for $r < 14/5$
$$\int^T_0 \iint \left(v_0^{\frac{\alpha}{2}-1} \left|\nabla_vf^{n+1}\right|\right)^r dvdxdt \leq C. $$
We then apply \eqref{Eq9} three more times.  In each application we take $h = w_0 \left|\nabla_vf^{n+1}\right|$.  
First, using 
$p < \frac{7}{6},\ q < 14/5$, and $\theta = \frac{\alpha}{2}-2$ yields 
$$\int^T_0 \iint \left(v_0^{\frac{\alpha}{2}-2}\left|\nabla_vf^{n+1}\right|\right)^r dvdxdt \leq C$$
for $r < 14/3$.  Using $p < 7/6,\ q < 14/3$, and $\theta = \frac{\alpha}{2}-3$ yields 
$$\int^T_0 \iint \left(v_0^{\frac{\alpha}{2}-3}\left|\nabla_vf^{n+1}\right|\right)^r dvdxdt \leq C$$
for $r < 14$.  Using $p= \frac{7}{6.4},\ q = \frac{14}{1.2},\ r= \infty$, and $\theta = \frac{\alpha}{2} -4$ yields
\begin{equation} \label{Eq17} 
\left|\left| v_0^{\frac{\alpha}{2}-4} \nabla_vf^{n+1}\right|\right|_{L^{\infty}} \leq C. 
\end{equation} 
Recall that $a > 8$ so $\frac{\alpha}{2}-4 = \frac{a+2 + \epsilon}{2}-4 > 1$.  Now $\forall h \in \mathbb{R}^2$
$$\left|f^{n+1}(t,x,v+h)-f^{n+1}(t,x,v)\right| \leq C|h|$$
and so 
$$\left| f(t,x,v+h) - f(t,x,v)\right| \leq C|h|. $$ 

Finally, we show that $f$ is H\"older continuous in $x$.  Let $h > 0$ and 
$$e = f^{n+1}(t,x+h,v) -f^{n+1}  (t,x,v), $$
then 
$$\begin{array}{lll} 
\partial_te+v_1 \partial_xe+K^n \cdot \nabla_ve-\Delta_ve\\
=- \left(K^n (t,x+h,v)-K^n(t,x,v)\right) \cdot \nabla_v f^{n+1}(t,x+h,v). \end{array} 
$$
Note that 
$$\begin{array}{llll}
\left|K^n(t,x+h,v)-K^n(t,x,v)\right|\\
\leq v_0 \left(\left|E^n(t,x+h)-E^n(t,x)\right|+ \left|B^n (t,x+h)-B^n(t,x)\right|\right) \\
\leq v_0h^{1/2} \left( \sqrt{\int \left(\partial_xE\right)^2dx} + \sqrt{ \int \left(\partial_xB\right)^2 dx}\right) \\
\leq Cv_0h^{1/2}. \end{array} 
$$
Thus, by \eqref{Eq17} we find
$$\left| \partial_te + v_1 \partial_x e+K^n\cdot \nabla_ve- \Delta_v e\right| \leq C v_0h^{1/2}Cv_0^{4-\alpha/2} \leq Ch^{1/2}.$$
By Lemma 2.1
$$\left|e\right| \leq Ch^{1/2}. $$ 
It follows that $f$ is H\"older continuous with exponent $1/2$ in $x$. 
Hence, Theorem II.1 of \cite{16} applies and shows that $f$ has the regularity stated in Theorem 1.1.  The regularity of $E$ and $B$ follows from this.

Finally, suppose that $(F, {\cal E}, {\cal B} )$ is another solution with the same initial value as $(f, E,B)$.  Then, by (\ref{E4.12}), (\ref{E4.13}), and (\ref{E4.16})

\be \begin{array}{rl}
& \dfrac{d}{dt} \diint v^b_0 (f-F)^2 \, dv dx \leq C \diint v^b_0 (f-F)^2 \, dv dx\\
\\
 & + C \dint \left( |E-{\cal E}|^2 + (B-{\cal B})^2\right) \dint F^2 v^{b+2}_0 \, dv dx \\
 \\
 \leq & C \diint v^b_0 (f - F)^2 dv dx  + C \dint \left(|E-{\cal E}|^2+(B-{\cal B})^2 \right)\, dx.
 \end{array}
 \label{E4.30}
 \ee

 \noindent Also

 $$\dfrac{d}{dt} \dint \left( |E-{\cal E}|^2 + (B - {\cal B})^2 \right) dx = -2 \dint (E-{\cal E}) \cdot \dint (f-F) \, v\, dvdx.
 $$

 \noindent Since

 $$
 \begin{array}{rcl}
 \left( \dint |f-F|v_0\, dv\right)^2 & \leq & \dint(f-F)^2 v^b_0 dv \dint v^{2-b}_0 dv \\
 \\
 & \leq & C \dint (f-F)^2 v^b_0 dv
 \end{array}
 $$

 \noindent we have

 \be
 \begin{array}{rcl}
 \dfrac{d}{dt} \dint \left( | E-{\cal E}|^2 - |B-{\cal B}|^2 \right)\, dx & \leq & C \sqrt{\dint|E-{\cal E}|^2\, dx}\ \sqrt{\diint(f-F)^2 v^b_0 \, dv dx}\\
 \\
 & \leq & C \dint |E-{\cal E}|^2 dx + C \diint (f-F)^2 v^b_0 \, dv dx.
 \end{array}
 \label{E4.31}
 \ee

 Uniqueness follows from (\ref{E4.30}) and (\ref{E4.31}) and the proof is complete.


\begin{thebibliography}{99}
 \bibitem[1]{1} Chae,  M., {\it The global classical solution of the Vlasov-Maxwell-Fokker-Planck system near Maxwellian},  Math. Models Methods Appl. Sci. 21, 5, 1007-1025, 2011.

 \bibitem[2]{2} Degond, P., {\it Global existence of smooth solutions for the Vlasov-Fokker-Planck equation in 1 and 2 space dimensions}, Ann. Sci. \^{E}cole Norm. Sup. 4, 19, 4, 519-542, 1986.

 \bibitem[3]{3} DiPerna, R. J., and Lions, P.-L., {\it Global weak solutions of Vlasov-Maxwell systems}, Comm. Pure Appl. Math. 42, 6, 729-757, 1989.

 \bibitem[4]{4} Felix, J., Calogero, S., and Pankavich, S., {\it Spatially homogeneous solutions of the Vlasov-Nordstrom-Fokker-Planck System}, J. Differential Equations  (to appear), 2014.

\bibitem[5]{5} Glassey, R., and Schaeffer, J., {\it On the ``one and on-half dimensional" relativistic Vlasov-Maxwell system}, Math. Methods Appl. Sci. 13,2, 169-179, 1990.

\bibitem[6]{6} Glassey, R. T., {\it The Cauchy problem in kinetic theory}, Society for Industrial and Applied Mathematics (SIAM), Philadelphia, PA, 1996.

\bibitem[7]{7} Glassey, R. T., and Strauss, W. A., {\it Singularity formation in a collisionless plasma could occur only at high velocities}, Arch. Rational Mech. Anal. 92, 1, 59-90, 1986.

\bibitem[8]{8} Lai, R., {\it On the one- and one-half dimensional relativistic Vlasov-Fokker-Planck-Maxwell system}, Math. Methods Appl. Sci. 18, 13, 1013-1040, 1995.

\bibitem[9]{9}  Lai, R., {\it One the one-and-one-half-dimensional relativistic Vlasov-Maxwell-Fokker-Planck syste with non-vanishing viscosity}, Math. Methods Appl. Sci. 21, 14, 1287-1296, 1998.

\bibitem[10]{10} Lions, P.-L., and Perthame, B., {\it Propagation of moments and regularity for the 3-dimensional Vlasov-Pisson system}, Invent. Math. 105, 2, 415-430, 1991.

\bibitem[11]{11} Pankavich, S., {\it Global existence for the Vlasov-Poisson  system with steady spatial asymptotics}, Comm. Partial Differential Equations 31, 1-3, 349-370, 2006.

\bibitem[12]{12} Pankavich, S., and Michalowski, N., {\it Global classical solutions of the one and one-half dimensional relativistic Vlasov-Maxwell-Fokker-Planck system}, (submitted), 2013.

\bibitem[13]{13} Pankavich S., and Michalowski, N., {\it A short proof of increased parabolic regularity},  (submitted), 2014.

\bibitem[14]{14} Pfaffelmoser, K., {\it Global classical solutions of the Vlasov-Poisson system in three dimensions for general initial data},  J. Differential Equations 95, 2, 281-303, 1992.

\bibitem[15]{15} Schaeffer, J., {\it Global existence of smooth solutions to the Vlasov-Poisson system in three dimensions}, Comm. Partial Differential Equations 16, 8-9, 1313-1335, 1991.

\bibitem[16]{16} Victory, H. D., and O'Dwyer, B. P., {\it On Classical Solutions of Vlasov-Poisson-Fokker-Planck Systems}, Indiana Univ. Math. J. 39, 1, 105-155, 1990.

\bibitem[17]{17}  Yang, T., and Yu, H., {\it Global classical solutions for the Vlasov-Maxwell-Fokker-Planck system}, SIAM J. Math. Anal. 42, 1, 459-488, 2010.
\end{thebibliography}
\end{document}